\documentclass[12pt,a4paper]{amsart}

\usepackage{amssymb, amstext, amscd, amsmath, amsfonts, amsthm, amscd, color}

\usepackage[mathscr]{euscript}  
 
\usepackage{tikz,mathdots,enumerate}

\usepackage{tikz-cd}

\usepackage{pgfplots}
\usepackage{url}
 
\newtheorem{thm}{Theorem}[section]

\newtheorem{cor}[thm]{Corollary}

\newtheorem{lemma}[thm]{Lemma}

\numberwithin{equation}{section}

\theoremstyle{definition}

\newtheorem{definition}[thm]{Definition}
\newcommand{\bR}{{\mathbb{R}}}

\newcommand{\bZ}{{\mathbb{Z}}}

  \newcommand{\C}{{\mathcal{C}}}
  \newcommand{\D}{{\mathcal{D}}}
  
  \newcommand{\F}{{\mathcal{F}}}

  \newcommand{\M}{{\mathcal{M}}}

\renewcommand{\S}{{\mathcal{S}}}
  \newcommand{\T}{{\mathcal{T}}}
  \newcommand{\U}{{\mathcal{U}}}

\usepackage{verbatim}

\usepackage{tikz}
\usetikzlibrary{calc}

\begin{document}

\setcounter{tocdepth}{1}

 \title[The generic minimal rigidity of a partially triangulated torus]{The generic minimal rigidity of a partially triangulated torus}

\author[J. Cruickshank, D. Kitson and S.C. Power]{J. Cruickshank, D. Kitson and S.C. Power}

\address{Dept.\ Math.\ Stats.\\ Lancaster University\\
Lancaster LA1 4YF \\U.K. }

\email{s.power@lancaster.ac.uk}

\email{james.cruickshank@nuigalway.ie}

\email{d.kitson@lancaster.ac.uk}

\thanks{2010 {\it  Mathematics Subject Classification.}
52C25. 05C25 \\
Key words and phrases: rigidity,  triangulated torus, triangulated surface\\
Partly supported by EPSRC grant  EP/J008648/1}

\maketitle

\begin{abstract}
A simple graph  is   $3$-rigid if its generic bar-joint frameworks in $\bR^3$ are infinitesimally rigid. Necessary and sufficient conditions are obtained for the minimal $3$-rigidity of a simple graph which is obtained from the $1$-skeleton of a triangulated torus by the deletion of edges interior to a triangulated disc.
\end{abstract}


\section{Introduction}
The graph of a triangulated sphere is generically $3$-rigid in the sense that any generic placement of the vertices in three-dimensional Euclidean space determines a bar-joint framework
which is continuously rigid.
This generic version of Cauchy's rigidity theorem for convex polyhedra follows from Dehn's determination \cite{dehn} of the infinitesimal rigidity of convex triangulated polyhedra. See also Gluck \cite{glu}.
In fact these graphs are minimally $3$-rigid (generically isostatic) in view of their flexibility on the removal of any edge. 

Generalising this, Fogelsanger \cite{fog} has
shown that a finite simple graph given by the $1$-skeleton of a triangulated compact surface without boundary is $3$-rigid.
The proof uses combinatorial edge contraction reduction of the graph together with the fact that $3$-rigidity is preserved by the inverse moves of vertex splitting. The methods also extend to higher dimensions. 
However, with the exception of the sphere the triangulated surface graphs are over-constrained, in the sense that $|E|> 3|V|-6$, and so it is natural to seek a combinatorial characterisation of minimal $3$-rigidity for the graphs of compact surfaces with boundaries.
We obtain such a characterisation here for graphs derived from torus graphs by the excision of the interior edges of a triangulated disc.
The precise definition of these graphs is given in Section \ref{s:surfacegraphs}. 

\begin{center}
\begin{figure}[ht]
\centering
\includegraphics[width=6cm]{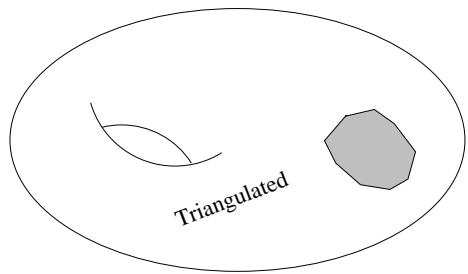}
\caption{A torus graph with a single superficial hole (shaded).}
\label{f:9v}
\end{figure}
\end{center} 

We make use of new methods, some of which have also been useful for modified triangulated spheres with holes and blocks \cite{cru-kit-pow}.
In particular we represent graphs in terms of \emph{face graphs},
which here are planar graphs with certain boundary vertices and edges identified,
and we consider reductions through edge contraction homotopy
and division over critical separating cycles.
However, in view of the toroidal topology there are 
a number of new considerations. In particular, edge cycles need not separate the graph and the boundary of the hole may be improper with subpaths wrapping around the torus with nontrivial homology. In fact for the minimally $3$-rigid graphs there are $17$ forms of hole boundary associated with which there are $17$ forms of critical separating cycle.
In addition we give a detailed analysis of  small  torus with hole graphs which have at most $9$ vertices and which in fact display all of these boundary types.

The main theorem may be stated as follows. 

\begin{thm}\label{t:mainthm}
Let $G$ be a torus graph with a single hole. 
Then the following conditions are equivalent.
\begin{enumerate} 
\item[(i)]$G$ is minimally $3$-rigid.
\item[(ii)] $G$ is $(3,6)$-tight.
\item[(iii)]  $G$ is constructible from $K_3$ by vertex splitting.
\end{enumerate}
\end{thm}

While $(3,6)$-tightness is a well-known necessary condition (see Section 3 for the definition) its sufficiency here is a more subtle issue than in the case of the generic Cauchy theorem. For example, the substitution of a triangulated subdisc by a triangulated disc with the same boundary need not preserve $(3,6)$-tightness. Also we note that there are torus graphs with two holes which are $(3,6)$-tight and generically  flexible.

The rest of the paper is concerned with two contrasting proofs of the main theorem and the introduction of methods which are likely to be useful for more general graphs associated with triangulations of compact surfaces with boundaries.

In Section 2  a \emph{torus graph with a single hole} is formally defined. 
In Section 3 we consider  the subfamily $\T$ of $(3,6)$-tight graphs $G$ of this type and we determine the $17$ forms of hole boundary together with  representative small graphs for them.
In Section 4 we define \emph{critical separating cycles} and associated \emph{fission moves} within the class $\T$. Exploiting the toroidal facial structure of the graphs in $\T$ we obtain
a key lemma, Lemma \ref{l:toruskey}, which shows that if the contraction of an edge $e$ preserves the simplicity of the graph but violates the $(3,6)$-tight sparsity count then there exists a critical separating cycle through $e$. In this case an associated fission move $G \to \{G_1, G_2\}$ is possible, which leads to a pair of strictly smaller graphs in $\T$ if $|V(G)|\geq 10$. It follows that there is a contraction-fission reduction scheme to
a certain family of small graphs in $\T$ with no more than $9$ vertices.

In Section 5 we prove that these small graphs are $3$-rigid. The 
inverse move for edge contraction is a vertex splitting move, which is known to preserve  $3$-rigidity (Whiteley \cite{whi-vertex}). Also the inverse fission moves, or fusion moves, correspond to rigid subgraph substitutions preserving $3$-rigidity, and so the equivalence of (i) and (ii) follows. 

In Section 6 we  give an alternative proof of this equivalence which is more direct. The proof is based on (i) a nested application of the key lemma, in order to identify a contractible edge whose contraction preserves membership in $\T$, and (ii)  an analysis of the  graphs of $\T$ which are not contractible in this manner. We show, moreover, that there are $2$ such uncontractible graphs (see Theorem \ref{t:theirreducibles}) and that each graph in $\T$ is constructible from (at least) one of these two graphs by a sequence of vertex splitting moves.

\section{Surface graphs in the torus}\label{s:surfacegraphs}
\label{Surfaces}
Let $\M$ be a compact surface, with or without a boundary. We define a 
\emph{surface graph for} $\M$ to be a simple graph obtained from the $1$-skeleton of a finite triangulation of $\M$. More generally we define  a \emph{surface graph} to be a graph $G = G(M)$ determined by the $1$-skeleton of a finite simplicial complex $M$ with the following properties.

\begin{enumerate}[(i)]
\item $M$ consists of a finite set of $2$-simplexes $\sigma_1, \sigma_2, \dots $ together with their $1$-simplexes and $0$-simplexes.
\item Every $1$-simplex lies in at most two $2$-simplexes.
\item $G(M)$ is simple.
\end{enumerate}

In particular, note that a surface graph $G$ is not merely a graph but is endowed with a \emph{facial structure} consisting of the set of $3$-cycles associated with the $2$-simplexes of  $M$. Also $G$ has  well-defined simplicial integral homology groups $H_i(G)= H_i(M,\bZ), i=1,2$.

A \emph{torus graph} is a surface graph $G$ for the torus $S^1 \times S^1$. This is a simple graph which may be obtained from a triangulated annulus graph by the identification of the inner and outer boundary cycles, where these cycles are assumed to have the same length and orientation. 
The triangulated annulus with its boundary identification gives  an \emph{annular face graph} representation for $G$.

We now define a torus graph with a single hole.

{\begin{definition}\label{d:torusgraphwithhole}
Let $M$ be the simplicial complex of a triangulation of a torus whose $1$-skeleton is a simple graph $T$. Let $D$ be  the simplicial complex of a triangulated disc and let $\iota$ be an injective map
from the set of $2$-simplexes of $D$ to the set of $2$-simplexes of $M$ which respects the adjacency relation between $2$-simplexes of $D$. Finally, let $G$ be the subgraph of $T$ obtained by deleting the edges associated with the $1$-simplexes which are images, under the map induced by $\iota$, of the interior $1$-simplexes of $D$. Then $G$ is said to be a \emph{torus graph with a single hole}.
\end{definition}}

We also refer to $G$ as a \emph{torus with hole graph} when there is no ambiguity.

It follows that a torus with hole graph is a simple graph which is  determined by a 
triple $(M, D, \iota)$ where 
$\iota : D \to M$ is a simplicial map from a simplicial complex $D$ of a triangulated disc to a simplicial complex $M$, of the torus, which is injective on $2$-simplexes. 
In particular the graph $G$ is endowed with the facial structure inherited from $M$.

It is convenient to abuse notation and let $D$, $T$ and $\partial D$ denote the graphs of the simplicial complexes $D$, $M$ and $\partial D$. Also we write $i$ for the associated simple graph homomorphism $i: D \to T$. 

The \emph{boundary graph} $\partial G$ of a torus with hole graph $G$ is the graph whose edges do not lie in two  facial $3$-cycles. Thus $\partial G$ is the image under $i$ of the boundary graph of $D$.

A torus with hole graph $G$ is also endowed with a specific $r$-cycle of edges, namely the image under $i$ of the boundary cycle of the graph $D$. This possibly improper $r$-cycle is determined by the restriction map 
\[
\alpha = i|_{\partial D} : \partial D \to \partial G,
\]
which we refer to as the \emph{detachment map} of $G$. 
This  map has the form
$\alpha : C_r \to \partial G$ where $C_r$ is the $r$-cycle graph, and is uniquely associated with $G$.

We now note some examples of torus with hole graphs.

Let $\M$ be the compact surface with boundary derived from $S^1\times S^1$ by the removal of the interior of a closed topological disc $\D$, where the topological boundary $\partial \D$ is a \emph{simple} closed curve. Then a surface graph for $\M$ is a torus with hole graph. These graphs correspond to the injectivity of the detachment map. Figure \ref{f:9v} illustrates such a graph.

On the other hand
Figure \ref{f:discontorus} indicates a torus graph with a single hole for which $\alpha$ is not injective and for which
\[
|V(\partial G)| = |V(\partial D)|-1, \quad  |E(\partial G)| = |E(\partial D)|
\]
\begin{center}
\begin{figure}[ht]
\centering
\includegraphics[width=5cm]{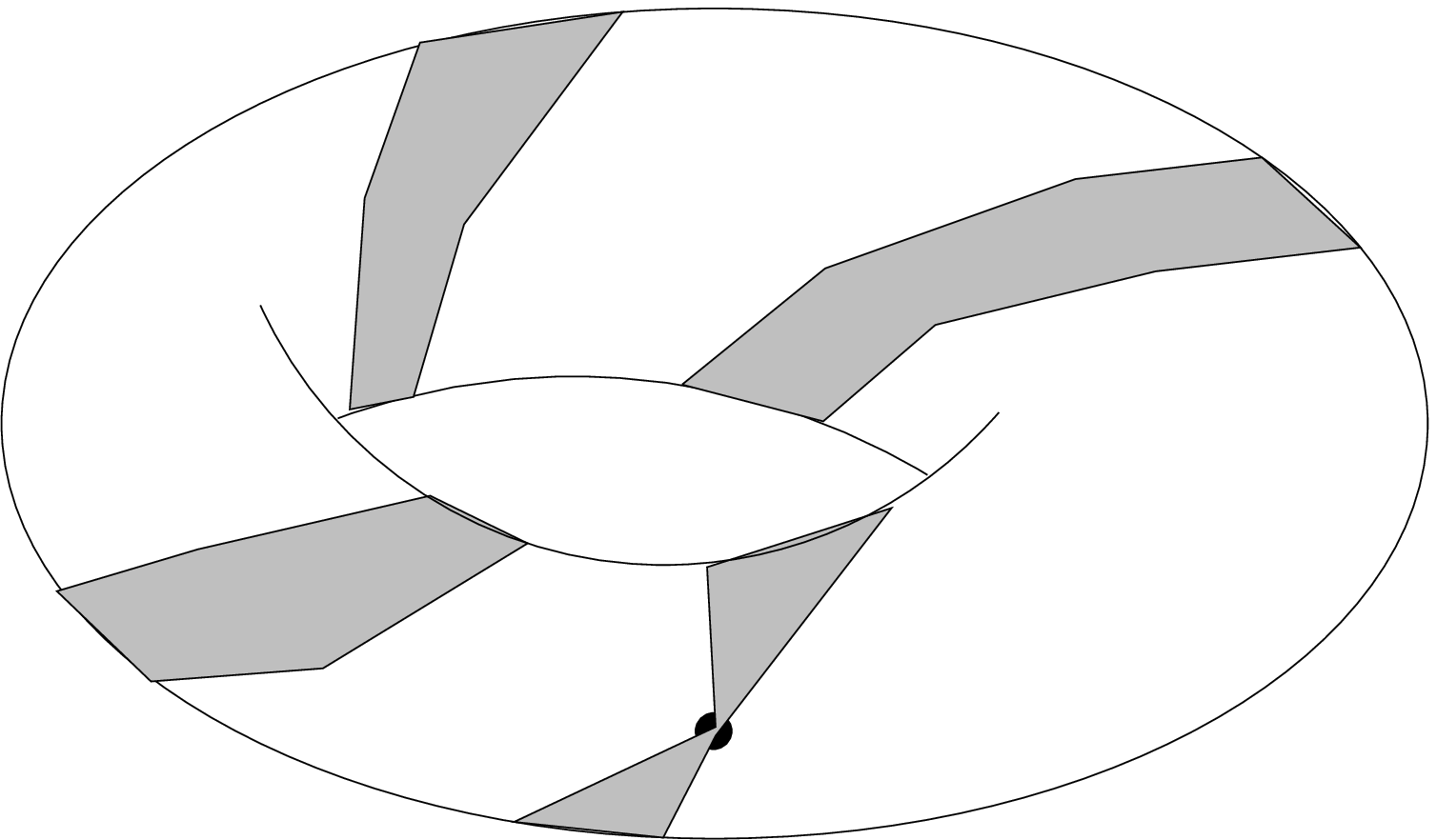}
\caption{A torus graph with a single hole (shaded). 
}
\label{f:discontorus}
\end{figure}
\end{center}  

Recall that a simple graph $G$ is $3$-connected if there exists no pair of vertices $x, y$ which separates the graph in the sense that there there are vertices $v, w$ such that each path from $v$ to $w$ contains one of the vertices in the pair. 
We note that a torus with hole graph may fail to be $3$-connected as indicated in Figure \ref{f:discontorusC}.
However, we shall see that the combinatorial condition of $(3,6)$-tightness, defined in the next section, limits the possible forms of noninjectivity of the detachment map. In particular these graphs are necessarily $3$-connected.

\begin{center}
\begin{figure}[ht]
\centering
\includegraphics[width=5cm]{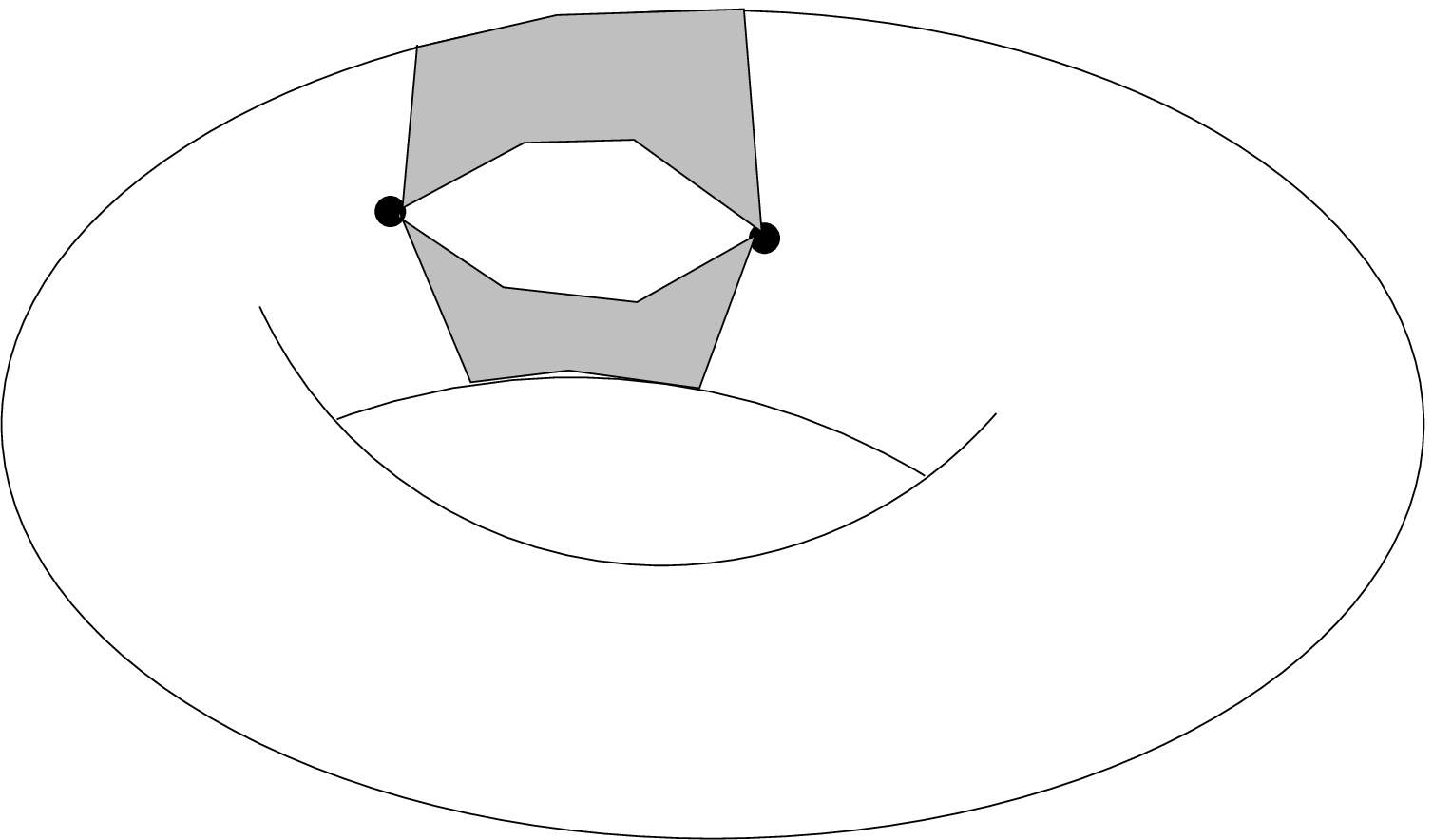}
\caption{A torus graph with a single hole (shaded).}
\label{f:discontorusC}
\end{figure}
\end{center}

On the other hand Figure \ref{f:3e4perspective} gives a perspective view of two torus with hole graphs with noninjective detachment map which 
can arise as $(3,6)$-tight graphs.
In the first figure the detached disk interior (the hole) wraps around the torus, and a single pair of vertices of $\partial D$ are identified. The graph of the second figure has an exposed edge corresponding to the identification of two edges of $\partial D$ under $\alpha$.

\begin{center}
\begin{figure}[ht]
\centering
\includegraphics[width=3.5cm]{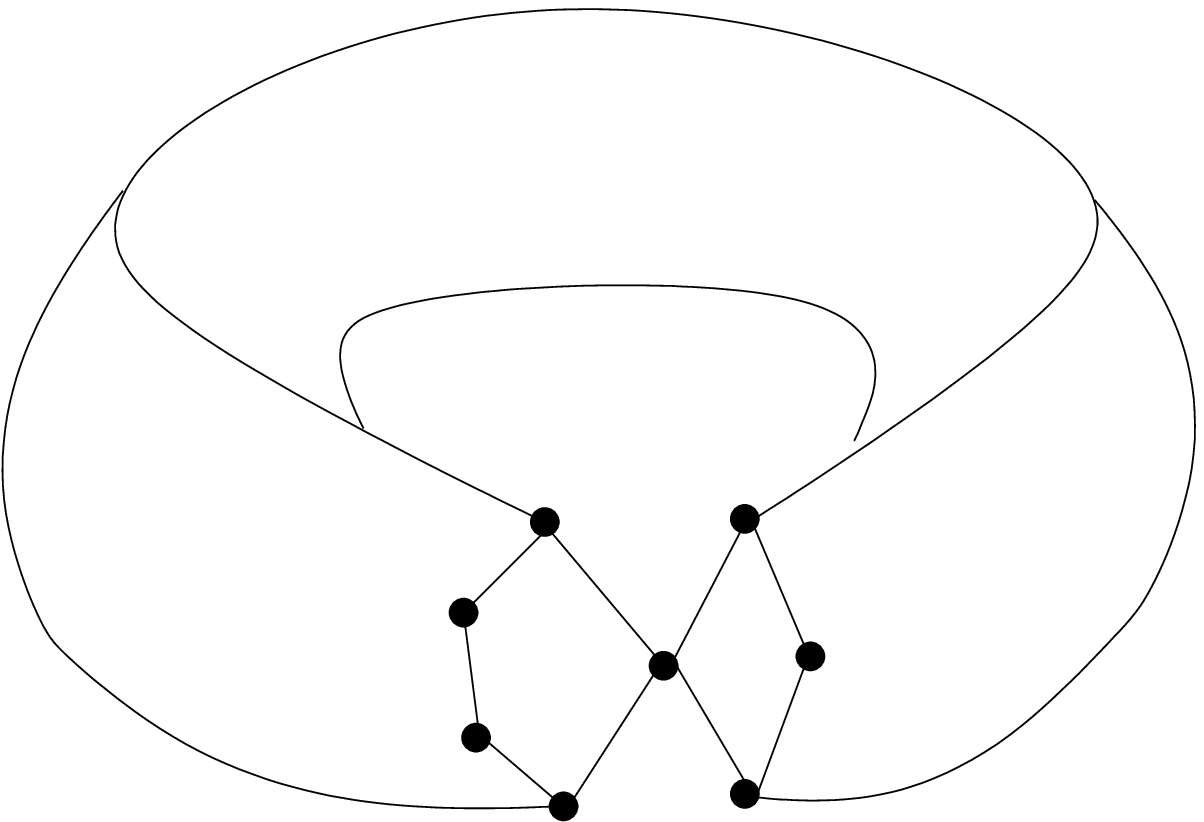}
\quad \includegraphics[width=3.5cm]{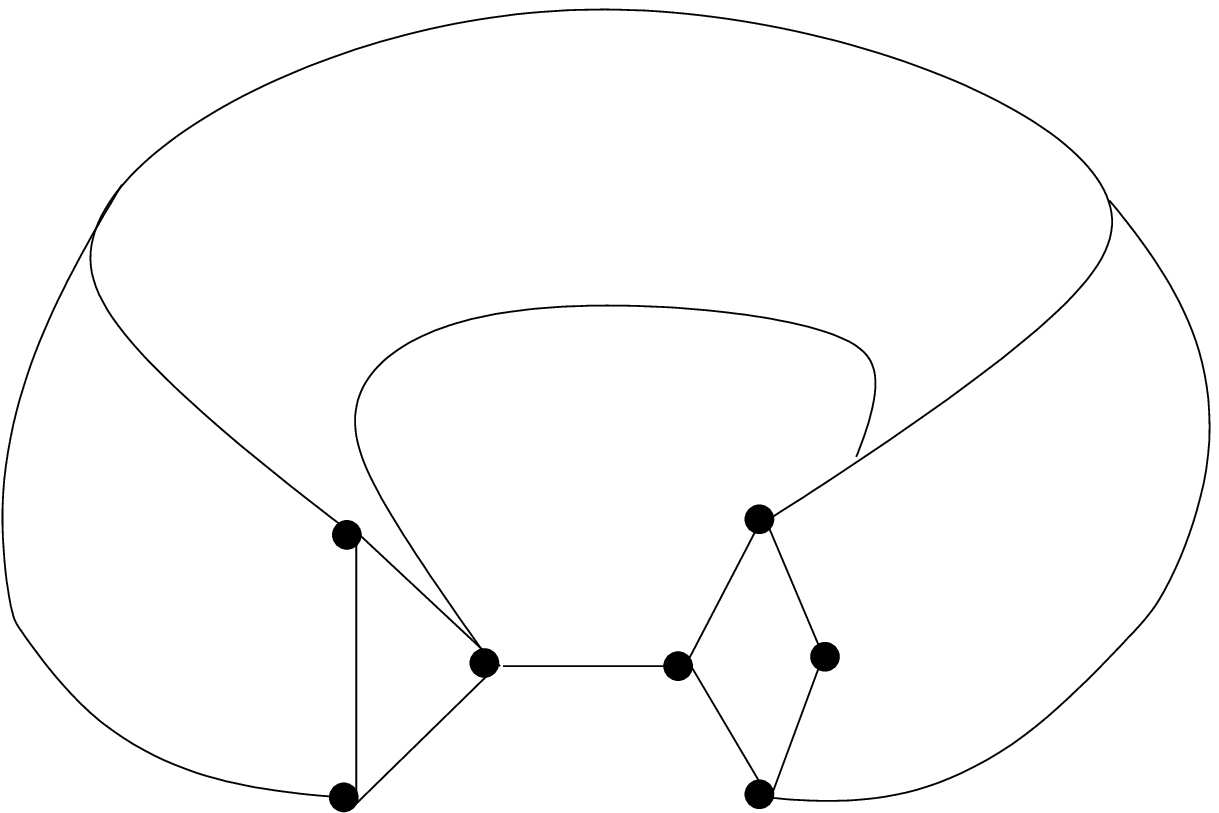}
\caption{Torus with hole graphs.}
\label{f:3e4perspective}
\end{figure}
\end{center} 
\bigskip

Finally, we remark that a torus graph $G$ with a single hole 
may take an extreme form with $\partial G =  G$. In particular $G$ may consist of two cycles of edges joined at a single vertex. Evidently such graphs are not $(3,6)$-tight.

\subsection{Planar representations}Torus graphs with a single hole may be given as the identification graphs of various planar partially triangulated graphs, with certain identifications of vertices and edges. We note two special forms of this.

Let $R$ be an annular face graph  for a torus graph $G$ and consider the graph obtained by deleting the edges and vertices of $R$ that are  interior to a \emph{proper} cycle $\pi$ of edges of $R$. Here $\pi$ is required to have distinct vertices and edges. In particular the boundary edges are not removed and the resulting graph $R'$, with the boundary matching for $R$, can be viewed as an  \emph{annular face graph}  for the associated identification   graph $G' = R'/\sim$. Such a graph is  evidently a torus graph with a single hole. Note that there may be edges on the hole boundary which are incident to no facial $3$-cycle (as in one of the graphs of Figure \ref{f:3e4perspective}).

We find it more convenient to illustrate a number of graphs by means of \emph{rectangular face graph} representations.
A rectangular face graph for a torus graph is a planar triangulated disc $R$ whose outer boundary path $\partial R$, as a directed cycle of edges, is a concatenation
$\partial R = \pi_1\pi_2\pi_3\pi_4$
where $\pi_1$ and $\pi_3$ are paths of length $r$,
$\pi_2$ and $\pi_4$ are paths of length $s$, and where these paths are appropriately matched. Formally this matching corresponds to appropriate bijections
$V(\pi_1) \to V(\pi_3)$ and $V(\pi_2) \to V(\pi_4)$ which are order reversing. The associated identification graph $R/\sim$ is a torus graph. Similarly, if $D$ is a triangulated disc in $R$, with boundary forming a proper cycle of (nonrepeating) vertices and $R'$ is obtained from $R$ by the removal of edges interior to $D$, then
the identification graph $G= R'/\sim$ is a torus graph with  a single hole. 

Additionally, it is useful to consider torus with hole graphs as embedded graphs on the topological torus, and we do this in the proofs of Lemmas \ref{l:fissionmove}, \ref{l:9vhas_stronglyinterior} and \ref{l:atmost9},  for example. Note in particular that a torus graph with a single hole has an embedded graph representation in a topological rectangular representation $R/\sim$ of the torus where, roughly speaking, the hole appears in the interior of $R$. 
More precisely, any torus graph with a single hole, with triple $(T,D,i)$, admits an embedded graph representation on a topological torus $\T$. This torus can in turn  be represented as the identification space $R/\sim$ of a closed rectangle $R$, where opposite edges are identified, and where the boundary of the rectangle arises from two simple closed paths in $\T$ which meet at a single point.  These paths may be chosen in the complement of the interior $U$ of the topological disc $\D$ of $D$, and it follows that $U$ corresponds (bijectively) to a subset of the interior of $R$.

\section{$(3,6)$-tight torus with hole graphs.}
We now consider torus graphs $G$ with a single hole which are $(3,6)$-tight and we determine the $17$ forms of the detachment map
$\alpha : C_r \to \partial G$ where $r=9$.
We remark that in the process of reduction  by edge contraction  or by critical separating cycle division, the hole boundary of the resulting smaller graphs may differ in form from the boundary  of $G$. For this reason, even in the special case of graphs whose boundary is a proper $9$-cycle, it is necessary to consider graphs with arbitrary detachment maps.

Recall that the {\em freedom number} for a finite simple graph $G=(V,E)$ is 
$f(G)=3|V|-|E|$ and that
if $f(G)=6$ then $G$ is said to satisfy the {\em Maxwell count}.

\begin{lemma}\label{l:boudaryis9cycle}
Let $G$ be a torus graph with a single hole determined by the triple $(T,D,i)$. Then $G$ satisfies the Maxwell count if and only if $\partial D$ is a $9$-cycle.
\end{lemma}

\proof 
The graph of a triangulated sphere satisfies the Maxwell count 
and so an annulus graph $A$ with two $r$-cycle boundary cycles has freedom number $f(A)=6+2(r-3)$. 
On the identification of the boundary cycles a total of $r$ vertices and $r$ edges are removed and the freedom number decreases by $2r$. Thus if $T$ is a torus graph then $f(T)=0$.
Note that 
\[
f(i(D))-f(i(\partial D))= f(D) - f(\partial D) = (|\partial D|+3)-2(|\partial D|=3-|\partial D|
\]
Now 
$
f(T)=f(G)+f(i(D))-f(i(\partial D))
$
and so $0=f(G)+3-|\partial D|$. Thus $f(G)=6$ if and only if $|\partial D|=9$.

\endproof

A simple graph $G$ is \emph{$(3,6)$-tight} if
$f(G)=6$ and $f(K)\geq 6$ for every subgraph $K$ with at least $3$ vertices. 
We write $\T$ for the class of torus graphs with a single hole which are $(3,6)$-tight.

Figure \ref{f:smallgraph} indicates a rectangular face graph representation for a graph $H_1$ in $\T$ for which the boundary graph is a proper $9$-cycle. 

\begin{center}
\begin{figure}[ht]
\centering
\includegraphics[width=3.5cm]{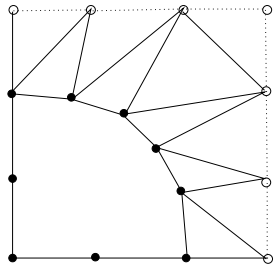}
\caption{A rectangular face graph for the graph $H_1$.}
\label{f:smallgraph}
\end{figure}
\end{center}

It is elementary to verify that $H_1$ is $(3,6)$-tight by means of the following principle. If a graph $H_a$ arises from a $(3,6)$-tight graph $H_b$ by vertex splitting, in the sense of the following definition, then $H_a$ is also $(3,6)$-tight. In this way it also readily follows that the graphs $H_2, \dots , H_{17}$ given below are graphs in $\T$.

Let $G=(V,E)$ be a simple graph with vertices $v_1,v_2,\dots ,v_r$  and let $v_1v_2, v_1v_3, \dots ,v_1v_n$ be the edges of   $E$ that are incident to $v_1$. Let $G'=(V',E')$ arise from $G$ by the introduction of a new vertex $v_0$, new edges $v_0v_1, v_0v_2, v_0v_3$, and the replacement of any number of the remaining edges $v_1v_t$, for $t > 3$,  by the  edges $v_0v_t$. Then the move $G \to G'$ is said to be a
\emph{vertex-splitting} move on $v_1$. The inverse of such a move is an edge contraction move which preserves simplicity.

The proof of rigidity preservation under vertex splitting is due to Whiteley \cite{whi-vertex}. A different proof, together with a proof of Gluck's theorem, is given in Cruickshank, Kitson and Power \cite{cru-kit-pow}.

\subsection{Graphs in $\T$ with noninjective detachment map}
When $\alpha$ is not injective the boundary of the hole can be regarded as a $9$-cycle which has been pinched together in some manner, with several self-contact points. The simplest form of this occurs when $|V(\partial G)| = 8$ and we  note that $\partial G$ then takes one of two forms, which we denote as $v3v6$ and $v4v5$. 
Figure  \ref{f:2v6and4v5} indicates two graphs $H_2$ and $H_3$ of these two types. 

In general a detachment map $\alpha : C_9 \to \partial G$ determines a vertex word of length $9$ whose letters are the vertices of $\partial G$, with repetitions, ordered in correspondence with the $9$-cycle of $C_9$.
This word is uniquely determined up to cyclic permutation and order reversal. However, we shall employ the economy of writing the short form $v3v6$ for the full form $v_1v_2v_3v_1v_4v_5v_6v_7v_8$ where $v=v_1$ is a repeated vertex. The short form should be read as "$v$ followed by $3$ distinct edges to $v$, followed by $6$ further distinct edges" (terminating in the first vertex $v$). An example of a cyclic word for a detachment map with $2$ repeated vertices (and no repeated edges) is $v1w2v2w4$. The numbers represent the $9$ distinct edges occurring in the $9$-cycle $i(\partial D)$ in this example. We  use the letters $v,w$ and also $x$ to denote repeated vertices.  When edges are repeated we adopt a more economical notation. For example $e3e4$ indicates that edge $e$ is repeated and that $e$ is followed by $3$ distinct edges then followed by $e$ (even though traversed in a different order) which is then followed by $4$ distinct edges to complete the cycle. We use the letters $e, f,$ and $g$ to denote repeated edges. In every cyclic word for the detachment map type the number of edge letters ($e, f$ or $g$), counted with multiplicities, together with the number of numerals, is equal to $9$.

\begin{center}
\begin{figure}[ht]
\centering
\includegraphics[width=5cm]{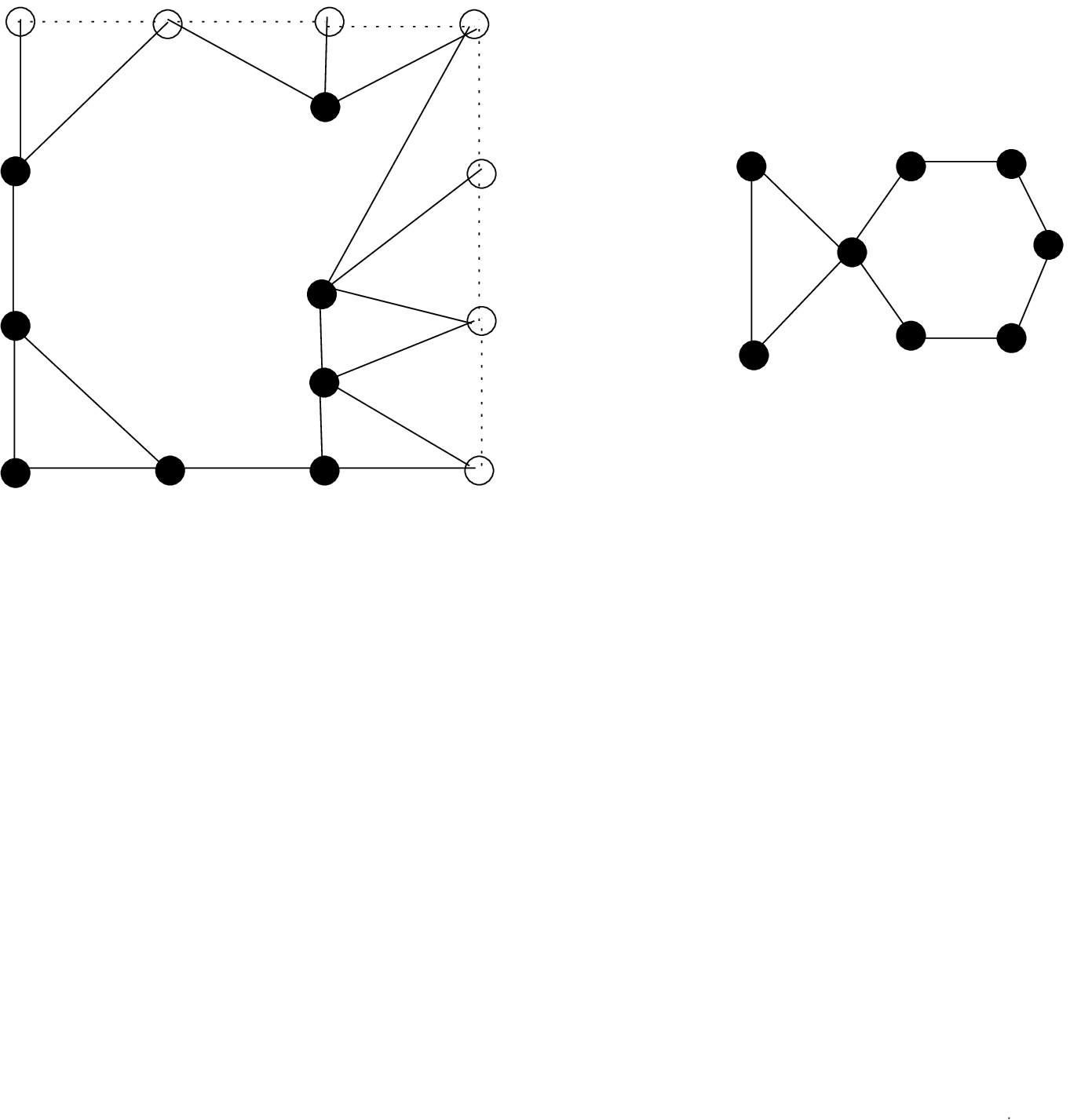}
\centering
\includegraphics[width=5cm]{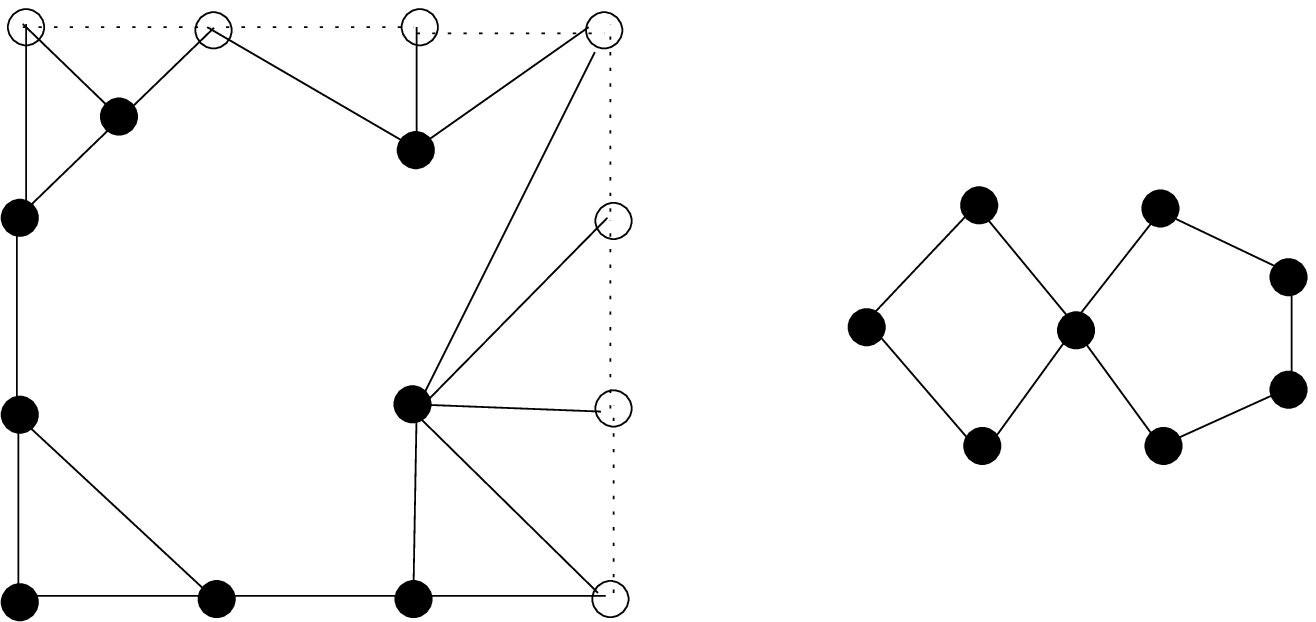}
\caption{Rectangular face graph representatives for the graphs $H_2$ and $H_3$ in $\T$  with boundary graphs of type $v3v6$ and $v4v5$.}
\label{f:2v6and4v5}
\end{figure}
\end{center}

Even the small graphs in $\T$, with $9$ or fewer vertices, form a surprisingly varied class as will become evident in the proof of the next lemma. The identification in this lemma of the precise nature of the detachment maps will also be useful in Section 6 for determining the uncontractible graphs in $\T$.


\begin{lemma}\label{l:smallgraphs}
There is a collection of graphs $ H_1, \dots ,  H_{17}$ in $\T$ with 
distinct detachment maps $\alpha_i, 1 \leq i \leq 17,$ and the following properties. 

(i) If $G \in \T$ with detachment map $\alpha$ then  $\alpha =\alpha_i$ for some $i$.

(ii) $V(H_i) = V(\partial H_i)$, for each $i$.
\end{lemma}

\begin{proof}
The proof of the lemma follows the following scheme. We identify the detachment maps $\alpha_1, \dots , \alpha_{17}$ that are possible
for graphs in the class $\T$, arguing case by case  for fixed values of $|V(\partial G)|$. These values range from $9$ to the minimum possible value which turns out to be $4$. At the same time we identify corresponding  vertex minimal graphs $H_1, \dots , H_{17}$ for these types.

Let $G\in \T$ be a $(3,6)$-tight torus with hole graph with the attachment map $\alpha$. The graphs $H_1, H_2, H_3$ and their detachment maps $\alpha_1, \alpha_2, \alpha_3$ have been described above. If $|V(\partial G)|=9$ then $\alpha = \alpha_1$ while if $|V(\partial G)|=8$ then either $\alpha = \alpha_2$ or $\alpha = \alpha_3$.

We next look at the case of $1$ repeated edge  which is the case $|V(\partial G)| = 7,   |E(\partial G)| = 8$. The detachment map, and the boundary graph, evidently has at most one form, with cyclic word $e3e4$. An associated vertex minimal  graph $H_4$ in $\T$ is defined by the rectangular face graph in Figure \ref{f:3e4}.

\begin{center}
\begin{figure}[ht]
\includegraphics[width=6cm]{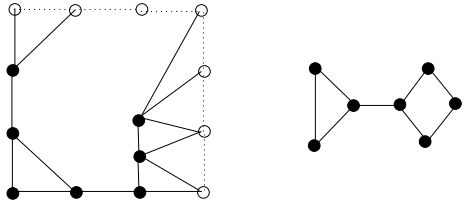}
\caption{A rectangular face graph representation for the graph  $H_4$, with boundary graph of type $e3e4$.}
\label{f:3e4}
\end{figure}
\end{center} 

We next consider the case of graphs in $\T$ with
\[
|V(\partial G)| = 7, \quad  |E(\partial G)| = 9
\]
Assume first that there are distinct vertices $v, w$ in $\partial G$ and $4$ edge-disjoint  consecutive paths of edges  between them, from $v$ to $w$, $w$ to $v$, $v$ to $w$ and $w$ to $v$, respectively, with lengths $a,b,c, d$ say, so that the cyclic word type for the detachment map of $\alpha$ is $vawbvcwd$. Thus we are assuming that $v$ and $w$ alternate in the cyclic word. Figure \ref{f:vvtype} indicates how this may be depicted as an embedded cycle on a topological torus.

\begin{center}
\begin{figure}[ht]
\centering
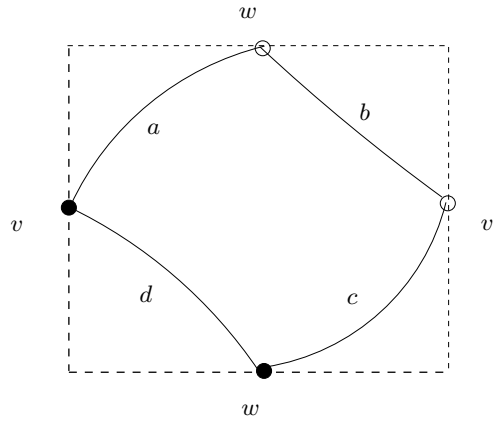
\caption{Boundary cycle of type $vawbvcwd$ with $a+b+c+d=9$.}
\label{f:vvtype}
\end{figure}
\end{center}

Since the graphs in $\T$ are simple it follows that the sum of any two of $a, b, c, d$  is at least $3$. In fact this is the only constraint and up to cyclic order and reversals there are $5$ types for such  quadruples $(a,b,c,d)$, namely
\[
(1,2,2,4), (1,2,3,3), (1,2,4,2),  (1,3,2,3), (2,3,2,2),
\]
and there are $5$ associated detachment maps, of types $v1w2v2w4$, etc.
The graphs in Figures \ref{f:H5H6H7}, \ref{f:H8H9} give representative vertex minimal graphs, $H_5, \dots , H_9$ 
for these attachment types.

\begin{center}
\begin{figure}[ht]
\includegraphics[width=3cm]{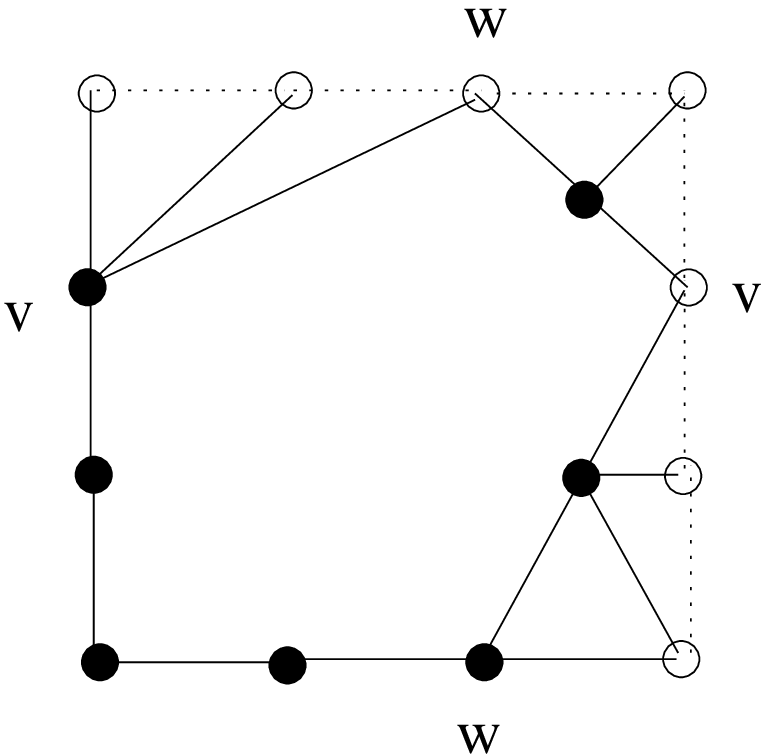}
\quad
\includegraphics[width=3cm]{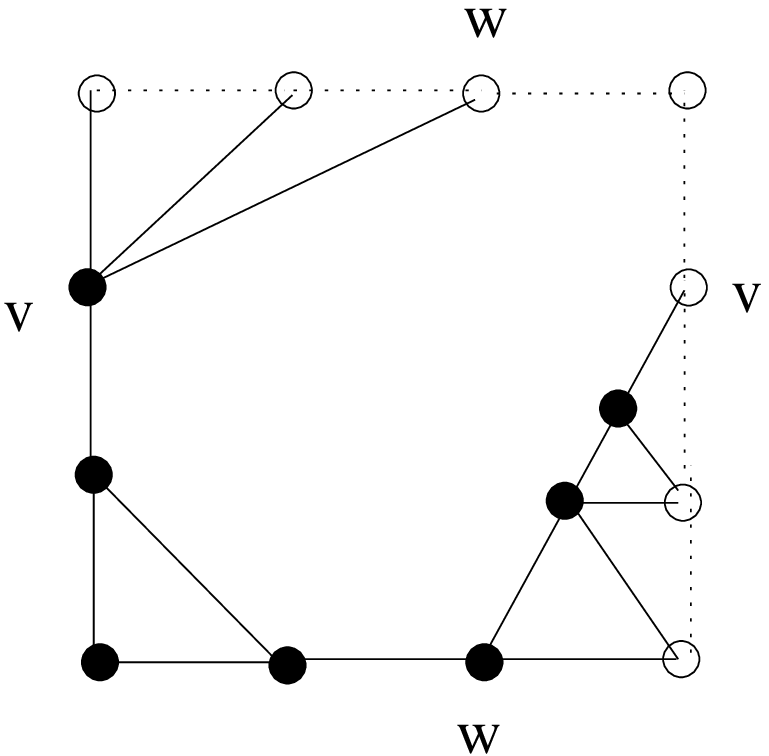}
\quad
\includegraphics[width=3cm]{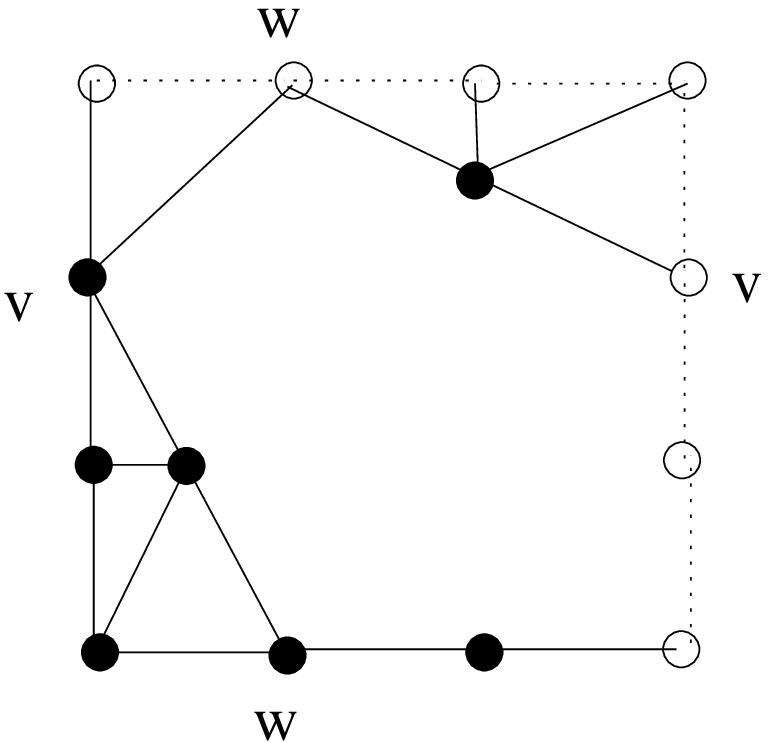}
\caption{Rectangular face graph representations for $H_{5}, H_{6}, H_{7}$.}
\label{f:H5H6H7}
\end{figure}
\end{center}

\begin{center}
\begin{figure}[ht]
\includegraphics[width=3cm]{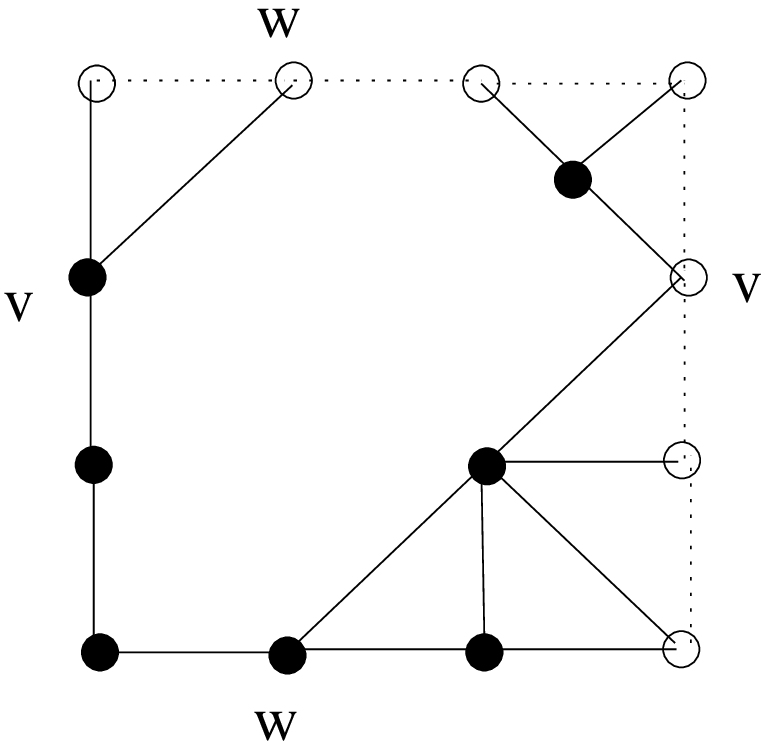}
\quad
\includegraphics[width=3cm]{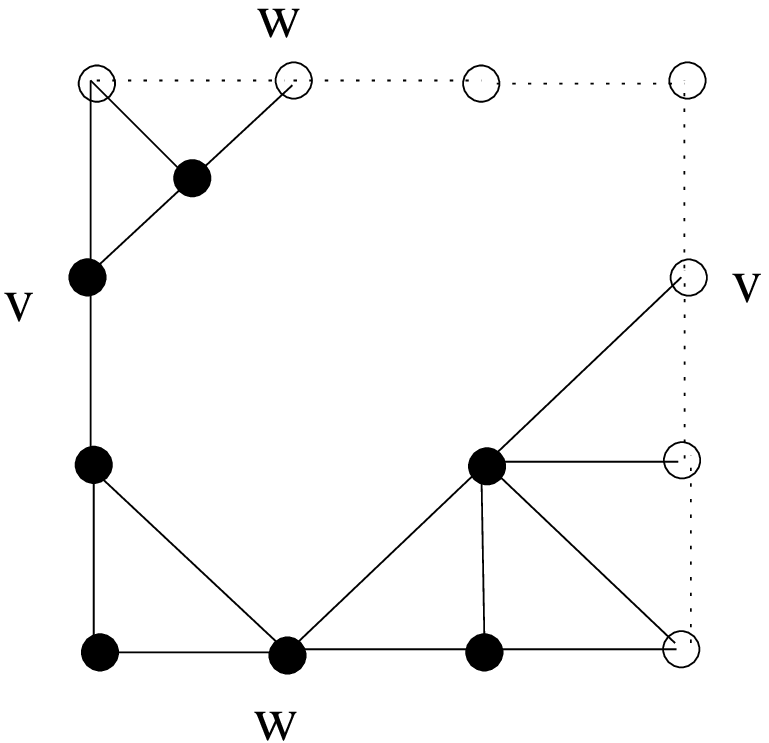}
\caption{Rectangular face graph representations for $H_{8}, H_{9}$.}
\label{f:H8H9}
\end{figure}
\end{center} 

We now look at the nonalternating case so that the cyclic word for the detachment map $\alpha$ is  $vavbwcwd$. Since $G$  is simple it follows that $a$ and $c$ are at least $3$ and so there is a unique form up to relabelling and ordering, namely $v3v2w3w1$. (See Figure \ref{f:vvtypeNotPoss}.)

\begin{center}
\begin{figure}[ht]
\centering
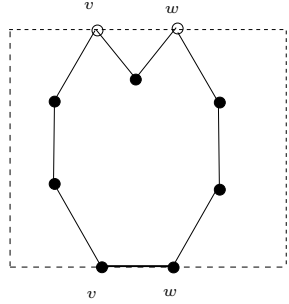
\caption{Boundary graph  type $v3v2w3w1$.}
\label{f:vvtypeNotPoss}
\end{figure}
\end{center}

In this case we see that $G$ contains a triangulated sphere which is formed by the part of the triangulated torus for $G$ which lies "between" the two nonfacial $3$-cycles associated with the subwords $v3v$ and $w3w$. Moreover $G$ contains the augmentation of this subgraph by the edge $vw$ and so $G$ cannot be $(3,6)$-tight.
Figure \ref{f:3v3v3} also indicates a perspective view of such a graph.
Since $G$ is assumed to be $(3,6)$-tight, this type of detachment map cannot occur.

\begin{center}
\begin{figure}[ht]
\centering
\includegraphics[width=4cm]{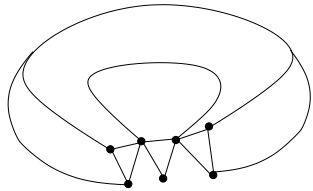}
\caption{A torus graph with a single hole satisfying the Maxwell count $f(G)=6$ which is not $(3,6)$-tight.}
\label{f:3v3v3}
\end{figure}
\end{center} 

Consider now the possibility of a detachment map with
$3$ (pairwise) repeated vertices and no repeated edges, so that $|V(\partial G)|=6$ and $|E(\partial G)|=9$. There are $2$ types of attachment map of types $v1w2x1v2w1x2$ and $ v1w1x1v2w2x2$
are represented by the graphs $H_{10}, H_{11}$ of  Figure \ref{f:vwxType}.
\begin{center}
\begin{figure}[ht]
\includegraphics[width=3cm]{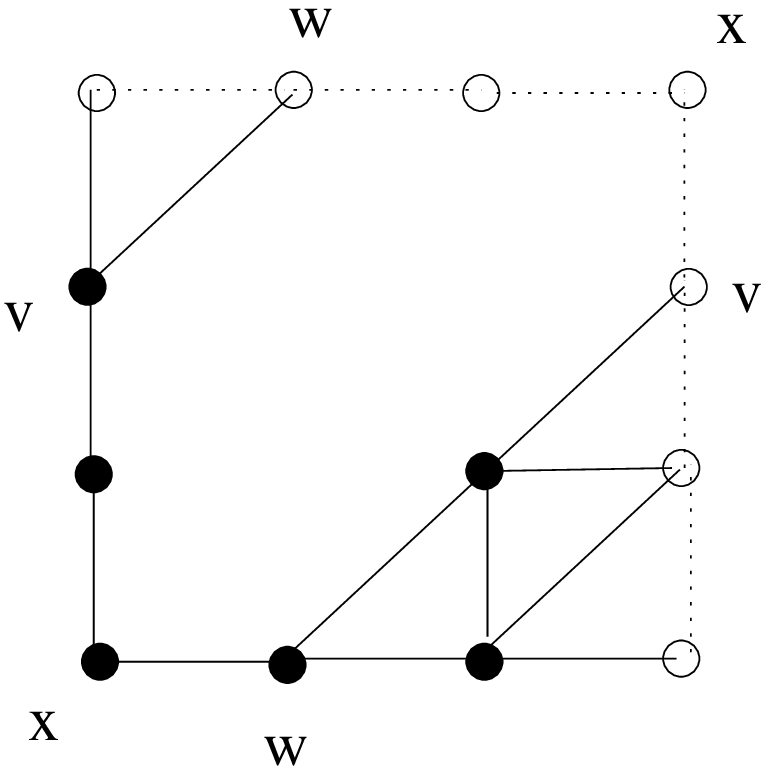}
\quad
\includegraphics[width=3cm]{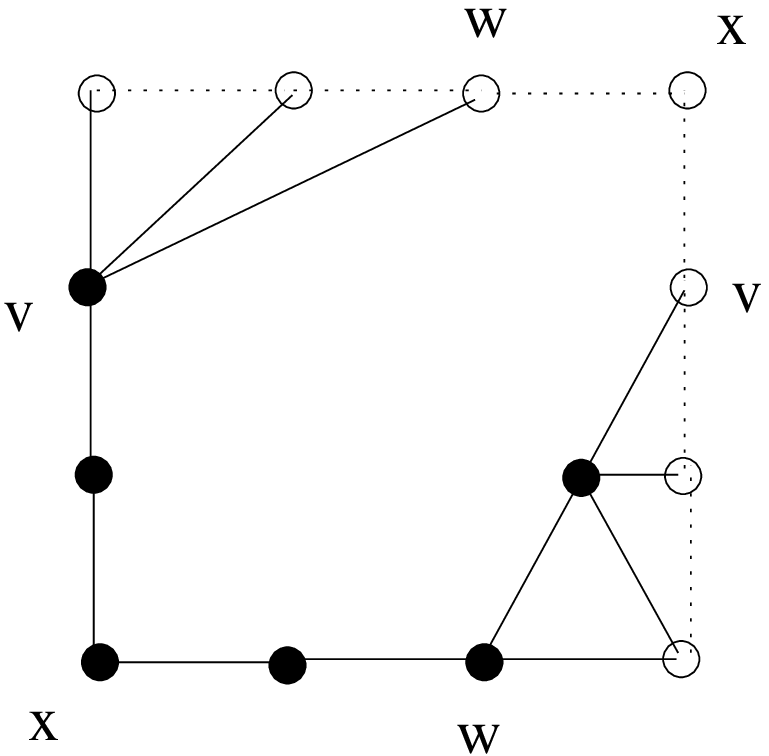}
\caption{Face graphs for $H_{10}, H_{11}$, with types $v1w2x1v2w1x2$ and $ v1w1x1v2w2x2$.}
\label{f:vwxType}
\end{figure}
\end{center} 
In fact there are no other forms possible for graphs in $\T$. Indeed, in analogy with the graph types of Figure \ref{f:3v3v3}, a torus with hole graph whose cyclic word contains  disjoint subwords of the form $w1v$ and $v2w$  is not $(3,6)$-tight.

We next consider further types with fewer than $9$ edges in the boundary graph. The first case to consider is 
\[
|V(\partial G)| = 6, \quad \quad  |E(\partial G)| = 8.
\]
In this case there are $3$ types of detachment map, for simple torus with hole graphs, and these correspond to the cyclic words
\[
v3e2v1e1, \quad v3e1v2e1, \quad  v2e2v2e1
\]
The structure of these words is depicted in Figure \ref{f:veType}
and representative vertex minimal graphs in $\T$ are given in Figure \ref{f:H12H13H14}.

\begin{center}
\begin{figure}[ht]
\centering
\includegraphics[width=6.5cm]{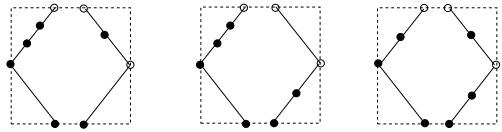}
\caption{Hole types $v3e2v1e1$ etc., for $H_{12}, H_{13}, H_{14}$}
\label{f:veType}
\end{figure}
\end{center}

\begin{center}
\begin{figure}[ht]
\includegraphics[width=3cm]{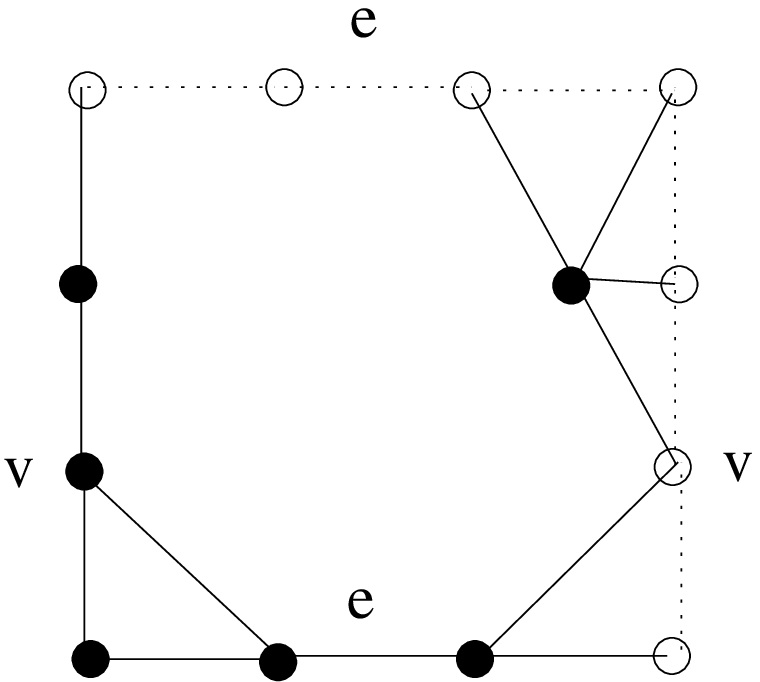}
\quad
\includegraphics[width=3cm]{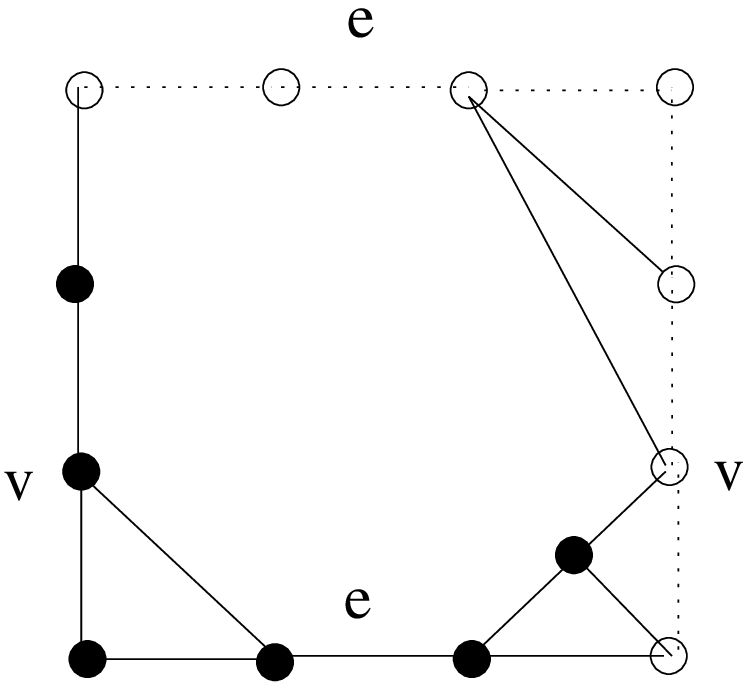}
\quad
\includegraphics[width=3cm]{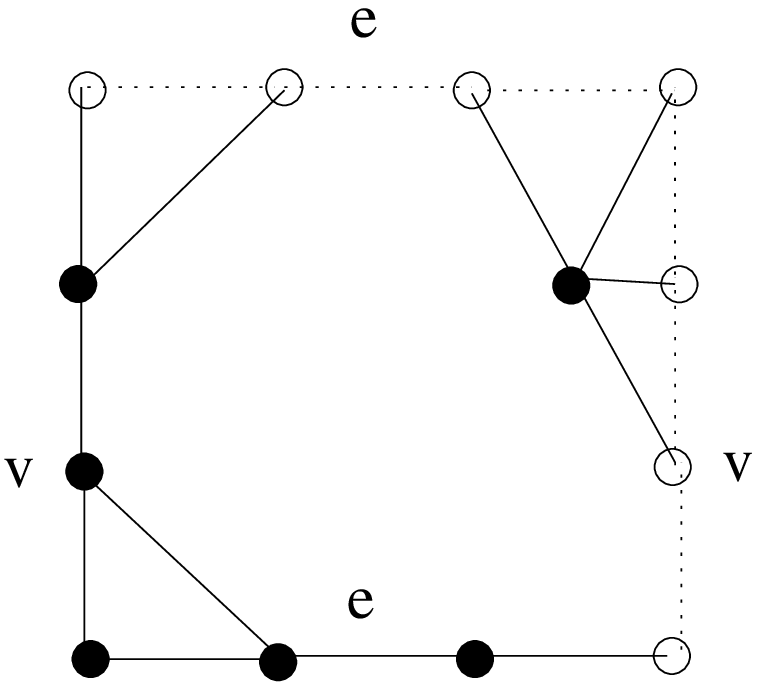}
\caption{Rectangular face graph representations for $H_{12}, H_{13}, H_{14}$.}
\label{f:H12H13H14}
\end{figure}
\end{center} 

The next cases are for torus with hole graphs  with
\[
|V(\partial G)| = 5, \quad  |E(\partial G)| = 8.
\]
Here the detachment map covers one edge of the boundary graph twice and two further vertices, $v, w$ are covered twice.
Such a graph in $\T$  is the graph $H_{15}$ in Figure \ref{f:k5minuseB}. This has cyclic type $v1e1w2v1e1w1$ and we note that the $v$ and $w$ vertices are alternating.
The only other possible cyclic type  is the nonalternating case   $v1e1v2w1e1w1$. Arguing as in the previous nonalternating case (depicted in Figure \ref{f:vvtypeNotPoss}) it follows that the graph $G$ cannot be $(3,6)$-tight. 

\begin{center}
\begin{figure}[ht]
\centering
\includegraphics[width=6.5cm]{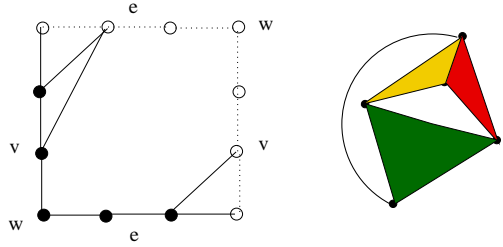}
\caption{The graph $H_{15}$ and its facial structure.}
\label{f:k5minuseB}
\end{figure}
\end{center} 

We next consider the case
\[
|V(\partial G)| = 5, \quad  |E(\partial G)| = 7.
\]
Note that up to relabelling and order there is one form of cyclic word, namely $e1f2e1f1$, and so one form of detachment map for graphs in $\T$.
Figure \ref{f:k5minuse} indicates  a vertex minimal representative, $H_{16}$, for this type.

\begin{center}
\begin{figure}[ht]
\centering
\includegraphics[width=6.5cm]{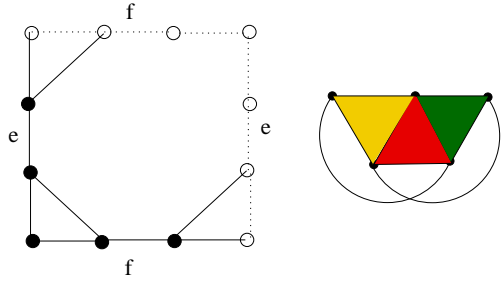}
\caption{The graph $H_{16}$ with detachment map type $e1f2e1f1$.}
\label{f:k5minuse}
\end{figure}
\end{center} 


Finally we consider the case of boundary graphs with $4$ vertices. There is one possible form of detachment map, with cyclic word $ef1ge1fg1$. A vertex minimal representative is given by the graph
$H_{17}$ in $\T$ in Figure \ref{f:H12}. 

\begin{center}
\begin{figure}[ht]
\centering
\includegraphics[width=6.5cm]{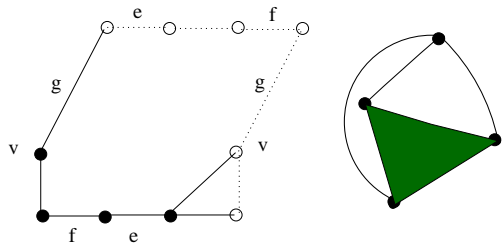}
\caption{The graph $H_{17}$ with detachment type $ef1ge1fg1$.}
\label{f:H12}
\end{figure}
\end{center} 

Note that $K_3$ is not a torus with hole graph $G$ arising from any triple $(T,D,i)$ where $D$ is  a $9$-cycle. 
To see this note that the map $i : \partial D \to \partial G$ must cover at least one edge of $\partial G$ more than twice. On the other hand each edge of $G$ is incident to at most $2$ faces, and $i$ is injective on faces, so this is not possible.
\end{proof}

We remark that the graphs $H_{16}$ and $H_{17}$ are uncontractible torus with hole graphs $G$ in the sense that there are no edges belonging to two faces in $G$ whose contraction yields a simple graph (and hence a torus with hole graph). On the other hand $H_1, \dots , H_{15}$ do have such edges, referred to as $FF$ edges in the next section.

We also remark that it follows from the main theorem that every graph in $\T$ is $3$-connected. However, this may also be proved directly by an embedded graph argument analogous to the one used in the exclusion of the small graph in Figure \ref{f:vvtypeNotPoss}.

\medskip
\medskip

\begin{center}
\begin{figure}[ht]
\centering
\begin{tabular}{|c|c|}
\hline
Cyclic word  &$G \in \T$   \\
\hline
\hline
 $v9$ &$H_1$   \\
\hline
$v3v6$&$H_2$    \\
\hline
 $v4v5$&$H_3 $   \\
\hline
 $e3e4$&$H_4 $  \\
\hline
$v1w2v2w4$&$H_5$    \\
\hline
$v1w2v3w3$&$H_6 $   \\
\hline
$v1w2v4w2$&$H_7$  \\
\hline
$v1w3v2w3$&$H_8$   \\
\hline
$v2w3v2w2$&$H_9$   \\
\hline
$v1w2x1v2w1x2$&$H_{10}$   \\
\hline
 $v1w1x1v2w2x2$&$H_{11}$   \\
\hline
$v3e2v1e1$ &$H_{12}$  \\
\hline
$v3e1v2e1$&$H_{13}$   \\
\hline
$v2e2v2e1$&$H_{14}$   \\
\hline
$v1e1w2v1e1w1$ &$H_{15}$  \\
\hline
$e1f2e1f1$&$H_{16}$   \\
\hline
$ef1ge1fg1$ &$H_{17}$
\\
\hline
\end{tabular}
\caption{The cyclic words that label the $17$ forms of detachment maps for graphs in $\T$ and a selection of associated vertex minimal graphs.}
\label{f:table}
\end{figure}
\end{center}

\subsection{Torus graphs with several holes}

There is an evident modification of Def.~\ref{d:torusgraphwithhole} which  defines a torus graph with \emph{several} (superficial) holes. We note the following two examples which are also $(3,6)$-tight.
Figure  \ref{f:rectangular1} shows a
rectangular face graph $R_1$ for a torus graph $G_1$ with two holes. Two triples of edges (dashed) have been deleted from the interiors of two triangulated discs in $R_1$.
The graph $G_1$ is $(3,6)$-tight and has a separating pair of vertices. In particular it is not $3$-rigid.
\begin{center}
\begin{figure}[ht]
\centering
\includegraphics[width=3cm]{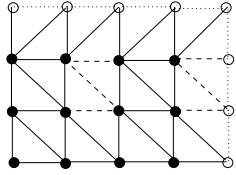}
\caption{The rectangular face graph $R_1$}
\label{f:rectangular1}
\end{figure}
\end{center}

Figure  \ref{f:rectangular2} indicates a  rectangular face graph $R_2$ for a torus graph $G_2$ with $6$ holes. The graph $G_2$ may be obtained from $G_1$ by the addition of two degree $3$ vertices and so $G_2$ is also $(3,6)$-tight and  fails to be $3$-rigid.
\begin{center}
\begin{figure}[ht]
\centering
\includegraphics[width=3cm]{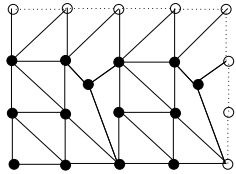}
\caption{The rectangular face graph $R_2$}
\label{f:rectangular2}
\end{figure}
\end{center} 

\section{Contraction moves in  $\T$ and critical separating cycles} 
Let $G$ be a torus graph with a single hole.
An edge of $G$ is of {\em type $FF$} if it is contained in two facial $3$-cycles and an $FF$ edge is {\em contractible} if it is not contained in any non-facial $3$-cycle, or, equivalently, if the contraction of the edge creates a simple graph (and thus a torus with hole graph). 
The contraction of a contractible $FF$ edge need not preserve $(3,6)$-tightness, and therefore membership in the class $\T$. However we shall show that when this occurs there exists a critical separating cycle and an associated graph division $G \to \{G_1, G_2^\circ\}$, where $G_1 \subseteq G$ is a graph in $\T$ and $G_2^\circ$ is a (possibly degenerate) annular graph whose two bounding cycles are $9$-cycles. By attaching  $G_2^\circ$ to the appropriate small graph $H_i$ (provided by Lemma \ref{l:smallgraphs}) whose detachment map agrees with that of $G_1$, we obtain a torus with hole graph $G_2$. Remarkably,  in all cases $G_2$ is simple and  $(3,6)$-tight, and  we refer to the resulting move
$G \to \{G_1, G_2 \}$ as a \emph{graph fission move} for the class $\T$.

\subsection{Critical separating cycles} 

Let $G$ be a torus graph with a single hole with triple $(T,D,i)$. Let $c$ be an $r$-cycle of edges in $G$. Then $c$ is said to be a \emph{hole separating cycle} or simply a \emph{separating cycle} if the context is clear, if there is a subgraph $G_1 \subseteq G$ which is a torus graph with a single hole with triple
$(T, D_1, i_1)$, where $D \subseteq D_1,$ and $c$ is the cycle $i_1(\partial D_1)$. In other words, a separating cycle is given by the image of the boundary of an enlargement $D_1$ of $D$ appearing in a commuting diagram of simplicial maps, 
\[
\begin{tikzcd}
D \arrow{r}{i} \arrow[hookrightarrow]{d} & T \arrow[rightarrow]{d}{{\rm id}} \\
D_1  \arrow{r}{i_1} & T
\end{tikzcd}
\]

A separating cycle gives a \emph{division move} $G\to \{G_1,G_2^\circ \}$ where $G_1$ is the torus graph with a single hole defined by a triple $(T,D,i_1)$, and $G_2^\circ$ is the complementary graph determined by the vertices of $G$ which are not interior vertices (nonboundary vertices) of $G_1$. 

The graph $G_2^\circ$ may be viewed as a triangulated annulus graph  with $2$ boundary cycles of length $9$ (which can nevertheless coincide at certain vertices and edges) where, in each cycle,  certain outer boundary vertices and edges may be identified. In particular $G_2^\circ$ need not be a planar graph when the separating cycle has double self contact.  This is the case for the critical separating cycle indicated in  Figure \ref{f:nonplanarannulus} for example. 


\begin{center}
\begin{figure}[ht]
\includegraphics[width=3cm]{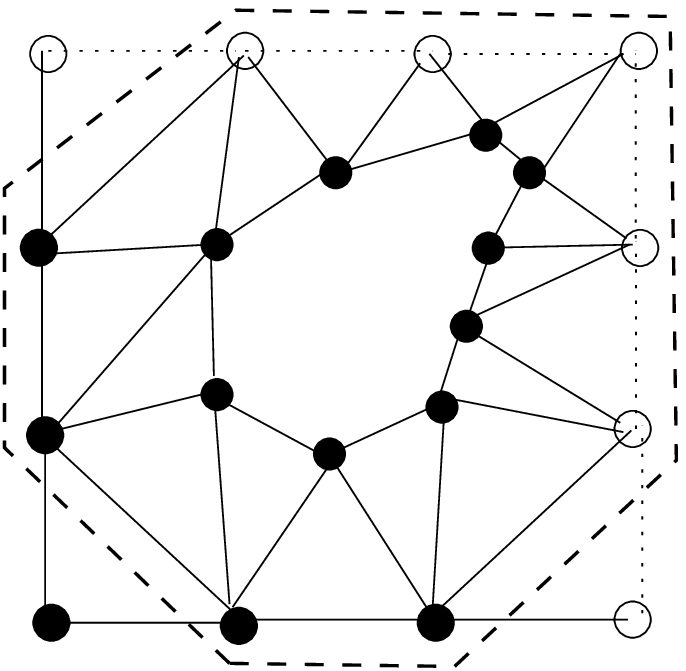}
\caption{A critical separating cycle with nonplanar annular graph $G_2^\circ$.}
\label{f:nonplanarannulus}
\end{figure}
\end{center}

\begin{definition}
Let $G \in \T$. Then a separating cycle for $G$ is a \emph{critical separating cycle} if the graph $G_1$ for the associated division move $G\to \{G_1, G_2^\circ\}$ is $(3,6)$-tight.
\end{definition}

The following ``filling in" lemma will be useful for the analysis of critical separating cycles. It may be paraphrased as the assertion that a $(3,6)$-tight subgraph of a graph $G\in \T$ contains no holes on its surface, bounded by $4$ or more edges, which do not contain the superficial hole of $G$. 

\begin{lemma}\label{l:fillingin}
Let $G\in \T$ and let $\tau$ be a (possibly improper) cycle of edges in $G$ which is given by the boundary cycle of the graph $H$ of an embedded triangulated disc in $G$. Let $K$ be a $(3,6)$-tight subgraph of $G$ with $K\cap H = \tau$. Then ${\tau}$ is a $3$-cycle.
\end{lemma}

\begin{proof}Note that the embedded assumption on $H$ is in the same sense as used for detachment maps, namely that $H$ is determined
by a simplicial map from a triangulated disc to the simplicial complex for $G$ with the property of being injective on $2$-simplexes. Let us write $H^c$ for the complementary graph to $H$
which contains $\tau$ and the edges of $G$ which are not in $H$.
Since $G = H^c\cup H$ and 
$H^c\cap H= \tau$ we have
\[
6=f(G) =f(H^c) +f(H)-f(\tau).
\]
Since $f(H^c)\geq 6$
we have 
$f(H)-f(\tau)\leq 0$.
On the other hand,
\[
6\leq f(K\cup H) =
f(K)+ f(H) -f(\tau)
\]
and $f(K)=6$ and so
it follows  that $f(H) -f(\tau)=0$.

If ${\tau}$ is a proper cycle then the identity
$f(H) -f(\tau)=0$  asserts that the freedom number of the proper boundary graph  $\partial H$ of the triangulated disc $H$ is equal to the freedom number of $H$ and this is only possible if the boundary is a $3$-cycle.
On the other hand, if $\tau$ is not a proper cycle and  $H_1/\sim$ is a simple graph, $H_2$ say,
obtained from a triangulated disc $H_1$ by the identification of some vertices and edges of $\partial H_1$, then 
the differences $f(H_2)-f(H_1)$ and  $f(\partial H_2)-f(\partial H_1)$ coincide. Thus it 
follows once more that $\tau$ is a $3$-cycle. 
\end{proof}

\subsection{Contraction and fission}
We now give a key lemma, Lemma \ref{l:toruskey}, which will be used for the deconstruction and construction of graphs in $\T$. The proof makes use of the following  topological property of certain open sets $U$ on the torus.

\begin{center}
\begin{figure}[ht]
\includegraphics[width=3.3cm]{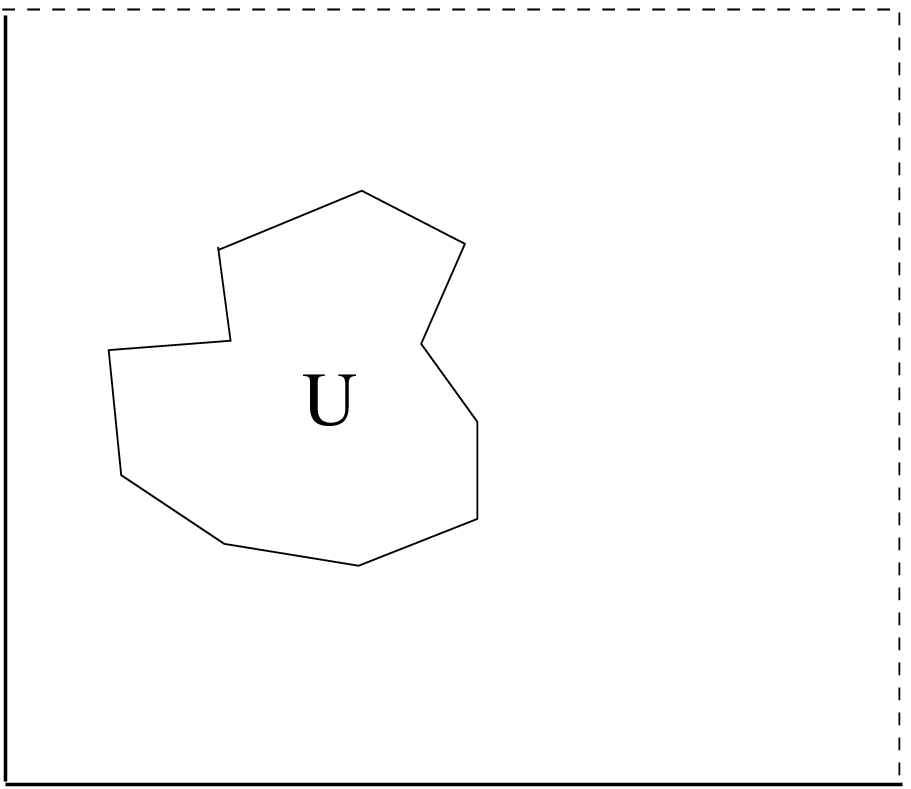}
\quad
\includegraphics[width=3.3cm]{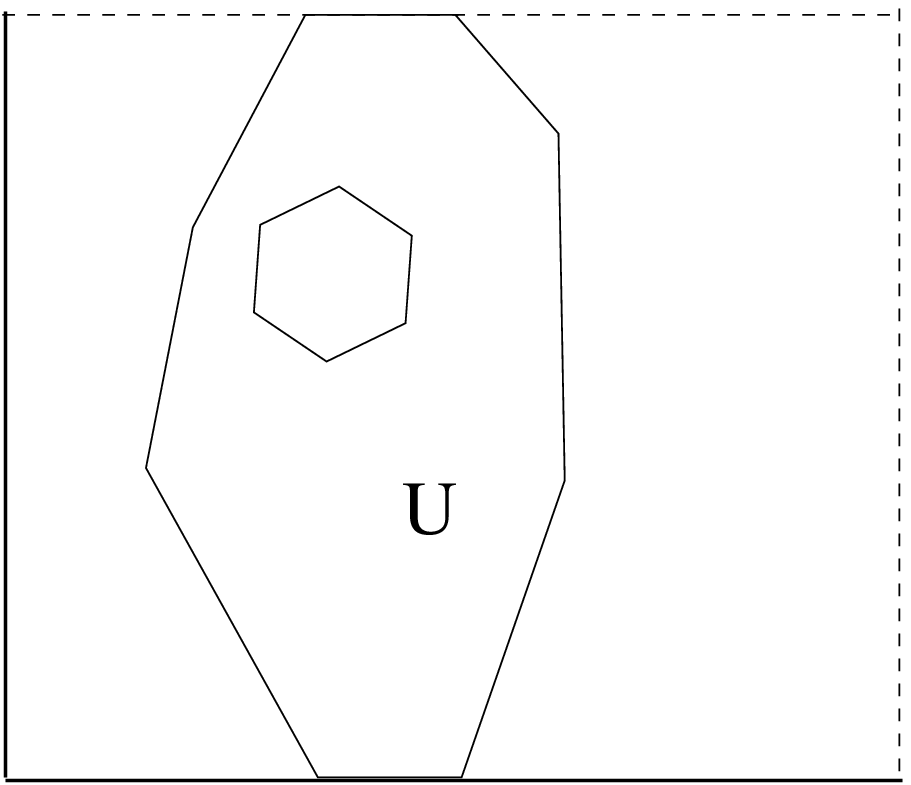}
\quad
\includegraphics[width=3.3cm]{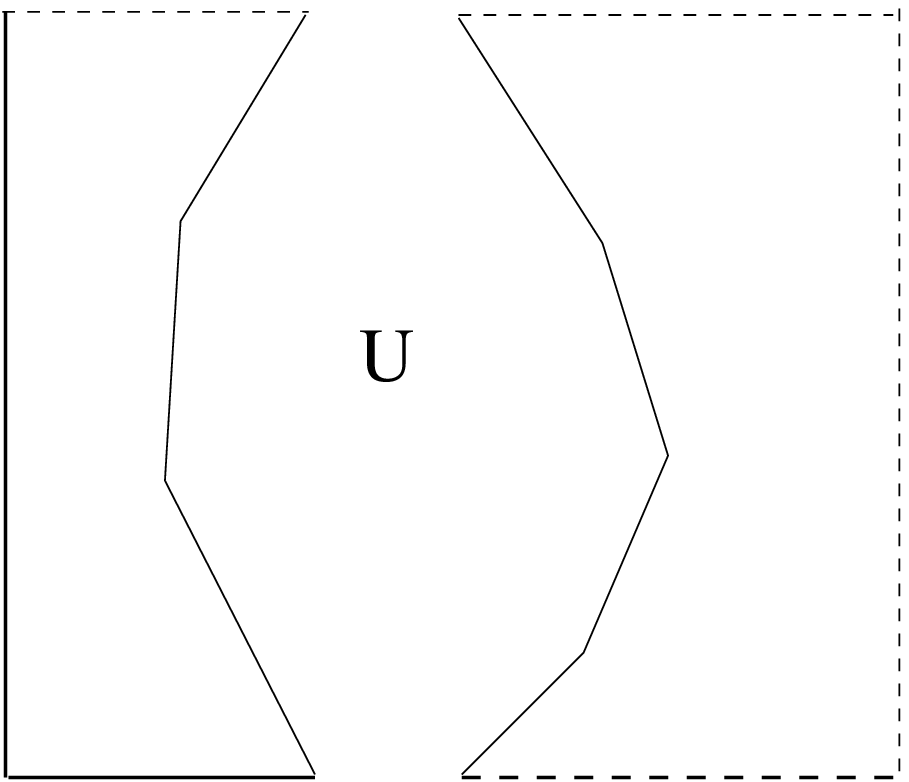}
\caption{Open sets $U$ in $S^1 \times S^1$ for the cases (i), (ii), (iii).}
\label{f:Ucases}
\end{figure}
\end{center}

\begin{lemma}\label{l:torusdichotomy}Let $\S$ be a family of (embedded) triangles
for a full triangulation of the topological torus $S^1\times S^1$.
Let $U$ be a connected open subset of  $S^1\times S^1$ which is the interior of the union of a subset $\F$ of closed triangles of  $\S$.
Then one of the following occurs.

(i)  $U$ is homeomorphic to an open disc.

(ii) The complement of $U$  is disconnected.

(iii) The complement of $U$  is homeomorphic to a closed subset of a $2$-sphere.
\end{lemma}

\begin{proof}
If $U$ is not homeomorphic to an open disc then there is a closed simple smooth path $\pi$ which is not homotopic in $U$ to a point. 

Suppose first that $\pi$ is homotopic to a point in $S^1\times S^1$. Then the complement of the range of $\pi$ in $S^1\times S^1$ has two open path-wise connected components, $V_1, V_2$, one of which, $V_1$ say, is homeomorphic to an open disc with boundary the range of $\pi$. If (ii) fails then the complement of $U$ can only meet one of the sets $V_1, V_2$. If it is $V_2$
then $V_1 \subseteq U$, contrary to the fact that $\pi$ is not homotopic in $U$ to a point. If it is $V_1$ then the complement of $U$ is contained in $V_1$ and so (iii) holds.

On the other hand if $\pi$
is not homotopic in $S^1 \times S^1$ to a point then it has nonzero homology class in $H_1(S^1\times S^1,\bZ)$. In this case a sufficiently small neighbourhood of (the range of) $\pi$ has a complementary set which is homeomorphic to a cylinder, and so (iii) follows.
\end{proof}

\begin{lemma}\label{l:toruskey}
Let $G\in\T$, let $e$ be a contractible $FF$ edge in $G$, and let $G'$ be the simple graph arising from the contraction move $G \to G'$ associated with $e$.
Then either $G'\in \T$ or the edge  $e$ lies on a (nontrivial) critical separating cycle.
\end{lemma}

\begin{proof} Let $(T,D,i)$ be a defining triple for $G$ and suppose that $G'\notin \T$. Note that the Maxwell count is preserved on contraction of the edge $e$ and so $G'$ must fail the $(3,6)$-sparsity count.
Thus there exists a subgraph $K$ of $G$ containing $e$ for which the edge contraction results in a graph $K'$ satisfying $f(K')<6$. 

Let $e=vw$ and let $c$ and $d$ be the facial $3$-cycles which contain $e$.
Note that if  both $c$ and $d$ are subgraphs of $K$ then 
$f(K)=f(K')<6$, which contradicts the sparsity count for $G$. Thus $K$ must contain either one or neither of these facial $3$-cycles.  

Suppose first that $K$ is a maximal subgraph among all subgraphs of $G$ which contain the cycle $c$ but not $d$ and for which contraction of $e$ results in a graph $K'$ which fails the $(3,6)$-sparsity count.
Note that $f(K)=f(K')+1$ which implies $f(K)= 6$ and $f(K')=5$.
In particular, $K$ is $(3,6)$-tight.

Let $\C$ be a maximal adjacency-connected collection of facial $3$-cycles in $T$ containing $d$, with the property that no facial $3$-cycle in $\C$ is a subgraph of $K$. Considering $T$ as embedded on the torus we claim that the interior $U$ of the union of the embedded triangles for the $3$-cycle faces in $\C$ has an interior set $U$ which is homeomorphic to the open unit disc. Indeed, if this were not the case then by the previous lemma there are two possibilities, namely cases (ii) and (iii). If (ii) holds then the complement of $U$ is not connected and therefore $K$ is not connected, since the boundary of $U$ meets the boundary of $K$. This contradicts $K$ being $(3,6)$-tight. If (iii) holds then $K$ and its facial $3$-cycles is embeddable on a $2$-sphere $S$. Since $e$ does not lie on a non-facial $3$-cycle of $G$ the triangulation of $K$ may be extended to a triangulation of $S$ with this property. This yields a contradiction since edge contraction of $e$ preserves the $(3,6)$-sparsity of such graphs.

Since $U$ is homeomorphic to an open disc it follows that the collection $\C$ determines an embedded triangulated disc  $\alpha(E)$ in $T$ for some triangulated disc $E$, where the graph morphism (or, more precisely, the simplicial morphism) $\alpha$ is injective on the faces of $E$.
By  the hole filling lemma, Lemma \ref{l:fillingin}, $\C$ cannot be disjoint from the set of facial $3$-cycles for $i(D)$. Since the facial $3$-cycles of $i(D)$ are absent from $K$ it follows from the definition of $\C$ (maximality) that $\alpha(E)$ contains
 $i(D)$. Thus the boundary cycle for $\alpha(E)$ is a critical separating cycle.

Suppose now that $K$ contains neither of the facial $3$-cycles which contain $e$. Then $f(K)=f(K')+2$ and so $f(K)\in\{6,7\}$.
Let $H_c$ be a maximal adjacency-connected collection of facial $3$-cycles in $T$ containing $d$ but not $c$ with the property that no facial $3$-cycle in $H_c$ is a subgraph of $K$.
Also let  $H_d$ be similarly defined for containment of  $c$ but not $d$, and let $H_0$ be similarly defined with containment of neither $c$ nor $d$. 
As before, by the maximality of these graphs and the hole-filling lemma  $H_c, H_d$ and $H_0$ contain the hole subgraph of $T$ for $G$. The subgraphs $H_c\cup H_0$ and $H_d\cup H_0$ determine two triangulated discs $D_1$ and $D_2$. By the hole-filling lemma again, we have $D_1^c\cap D_2^c=K$. Also $G= D_1^c\cup D_2^c$ and so
\[
6=f(G)=f(D_1^c)+f(D_2^c)-f(K)
\]
and so either $D_1^c$ or $D_2^c$ is $(3,6)$-tight. 
It follows that either $\partial D_1$ or $\partial D_2$ is a critical separating cycle which contains $e$. 
\end{proof}

We now show that  the move $G \to \{G_1, G_2\}$, associated with a critical separating cycle, is indeed a fission move in the class $\T$.

\begin{lemma}\label{l:fissionmove}
Let $G \to \{G_1,G_2^\circ\}$ be a division move associated with a critical separating cycle whose detachment map has associated graph $H_i$. Let $G_2=H_i \cup G_2^\circ $ be the torus with hole graph obtained by substituting $H_i$ for $G_1$ in the graph $G$. Then $G_2$ is  simple and $(3,6)$-tight.
\end{lemma}

\begin{proof}
Let $A$ be the (possibly degenerate) annulus graph $G_2^\circ$. Then the graph intersections $G_1\cap A$ and $H_i \cap A$ coincide. Also
we have
\[
6=f(G)=f(G_1\cup A) = f(G_1)+f(A)-f(G_1\cap A)=6+(f(A)-f(H_i\cap A))
\]
Thus $f(A)-f(H_i\cap A) = 0$ and so
\[
f(G_2) = f(H_i)+f(A)-f(H_i\cap A) = 6.
\]

It remains to show that $G_2$ is simple and  $(3,6)$-sparse.

If $G_2$ is not simple then there is A non boundary edge $e$ of $H_i$ with the same vertices as an edge $f$ of $A$. But then $G_2\cup\{f\}$ is a subgraph of $G$ with freedom count $5$ which is a contradiction.

To determine the sparsity of $G_2$ let $K$ be a subgraph of the graph $G_2=H_i\cup A$ with at least $3$ vertices. Then 
\[
f(G_1\cup (K\cap A)) = f(G_1) +f(K\cap A)-f(G_1\cap A\cap K)
\]
from which it follows that $f(K\cap A)-f(G_1\cap A\cap K)\geq 0$, since $G$ is $(3,6)$-sparse and $f(G_1)=6$. On the other hand
\[
G_1\cap A\cap K = H_i\cap A\cap K
\]
and so
\[
f(K)=f(K\cap H_i) + f(K \cap A) - f(H_i\cap A\cap K)\geq f(K\cap H_i)
\]
Thus $f(K)\geq 6$ (as desired) except possibly in the case that $K\cap H_i$
consists of a single edge. If this edge is an edge of the boundary cycle then $K$ is a subgraph of $G$ and it follows immediately that $f(K)\geq 6$. 
So the final case to consider is the case $K=K_1\cup e$ where $K_1$ is a $(3,6)$-tight subgraph of $A$ which meets $H_i$ at $2$ vertices, being the vertices of a nonboundary  edge $e$  of $H_i$. We use the hole filling lemma to show that this does not occur.

We may consider an embedded graph representation of $G_1$  in a topological torus represented by a rectangle with identified opposite edges.
Moreover we may assume that 
the critical cycle $c$ for $G_1$ is represented by a simple closed curve marked with vertices of $c$. It follows that
$K$ is represented as an embedded triangulated planar graph, possibly with some identified vertices or edges (implied by the detachment map for $G$). Since there are $2$ vertices of attachment, which we denote by $x$ and $y$, it follows that the interior region of the critical cycle $c$ contains at least two open disjoint regions complementary to $K$, one of which corresponds to the detachment map (or hole) of $G$. 
 (There may, a priori, be more than one non-hole region, as suggested by the regions labelled $A$ and $B$ in Figure
\ref{f:2vertexCaseOfK}.)

\begin{center}
\begin{figure}[ht]
\centering
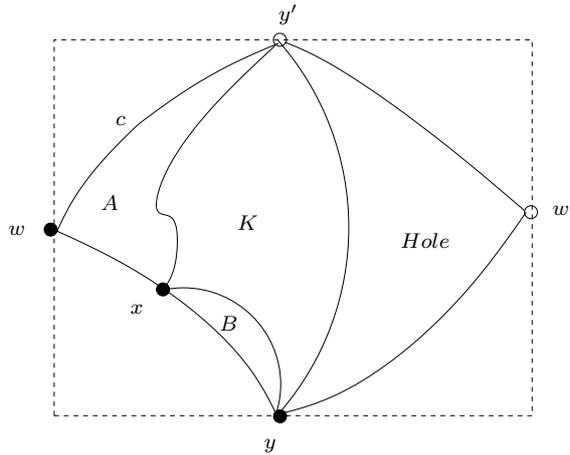
\caption{A torus embedded graph indication of a subgraph $K$ of $G$ with two vertices of attachment to $G_1$.}
\label{f:2vertexCaseOfK}
\end{figure}
\end{center}

By the hole filling lemma,  applied to the $(3,6)$-tight graph $G_1\cup K$, the non-hole regions are bounded by embedded $3$-cycles. These are either
facial $3$-cycles or nonfacial $3$-cycles which are triangulated by faces of $G$. We now obtain a contradiction in all cases. 

Consider a path from $y$ to $x$ along the boundary of $K$ for one of the non hole regions.
This has length $1$ or $2$. If it has length $1$ then adding this edge to $G_1$ gives a subgraph of $G$ with freedom count $5$ which is a contradiction.
Thus the length of the path is $2$ and so the length of the path along $c$ from $x$ to $y$ has length $1$. Adding this edge to $K$ gives a subgraph of $G$ with freedom count $5$ and so this again is a contradiction.
\end{proof}

We now deduce that there is a \emph{contraction fission sequence} for any graph $G\in\T$ as described in the next corollary.

Note that critical separating cycles include the improper case of the boundary cycle determined by the triple $(T,D,i)$ for $G$ and in this case
$G_2^\circ$ is equal to $\partial G$. 
Also we note that the proper critical separating cycle $c_1$ in Figure \ref{f:nonplanarannulus} does not provide a fission move with the property, which we call the \emph{reducing property}, that $G_2$ has a smaller vertex set than $G$. 
However, we now show that we can assume that there is a critical cycle for which reduction occurs. The corollary then follows from this fact and the key lemma.

Let $c$ be a proper critical separating cycle for $G\in \T$. Then 
$G_2^\circ$ contains a face of $G$ with an edge  $xy$ on $c$ and third vertex $z\in G_2^\circ$. It follows that both of the edges $xz$ and $zy$ do not lie on $c$ since adding an edge to $G_1$ provides a subgraph of $G$ with freedom number $5$. Thus there is an $FF$ edge in $G_2^\circ$ with both faces in $G_2^\circ$ . If contraction of this edge yields a graph in $\T$ then this may be used for reduction. So we may assume, by the key lemma, that there is a critical separating cycle $c_2$ though this $FF$ edge. We can also assume that this cycle lies in $G_2^\circ$ (by replacing some subpaths with corresponding subpaths of $c_1$ with the same initial and final vertices). It now follows that the division move for $c_2$ produces an annular graph with fewer faces than $G_2^\circ$. The resulting annular graph need not in fact have fewer vertices than $G$. Nevertheless, since the graph is finite,  the process can be repeated until a proper cycle is obtained with this property.

\begin{definition} A torus with hole graph is \emph{uncontractible} if every edge of type $FF$ lies on a nonfacial $3$-cycle. 
\end{definition}

\begin{cor}\label{c:tree}
\label{c:canconstruct} 
Let $G$ be a torus graph with a single hole in $\T$. 
Then there exists a finite rooted tree in which each node is labelled by an element of $\T$ such that,
\begin{enumerate}[(i)]
\item the root node is labelled $G$,
\item every node has either one child which is obtained from its parent node by an $FF$ edge contraction, or, two children which are obtained from their parent node by a fission move for a critical separating cycle, 
\item each leaf is an uncontractible graph.
\end{enumerate}
\end{cor}

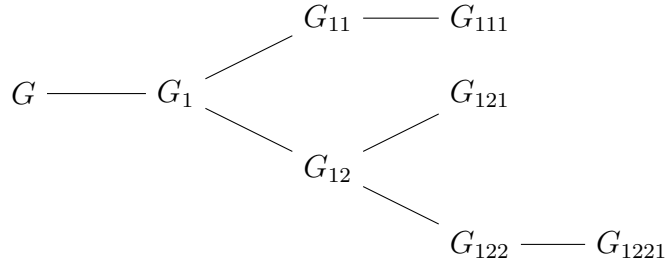
\begin{figure}[ht]
\centering
  \begin{tabular}{  c   }
\begin{tikzpicture}[level distance=2cm,
  level 1/.style={sibling distance=2cm},
  level 2/.style={sibling distance=2cm},
	level 3/.style={sibling distance=2cm},
	level 4/.style={sibling distance=2cm}]
  \node {$G$} [grow=right]
    child {node {$G_1$}
      child {node {$G_{12}$}
				child{node {$G_{122}$}
				child{node {$G_{1221}$}}
				}
				child{node {$G_{121}$}}
			}
      child {node {$G_{11}$}
				child{node {$G_{111}$}}
      }};
\end{tikzpicture}
\end{tabular}

 \caption{Contraction and fission to uncontractible graphs in $\T$.}
\label{ConstructionA}
\end{figure}

The inverse move for edge contraction is a vertex splitting move, which as we have noted, preserves $3$-rigidity. Also the inverse of a fission move (a fusion move) corresponds to substitution of the subgraph $H_i$ of $G_2$ by the graph $G_1$. It is immediate from the definition of infinitesimal rigidity that if $G_1$ and $G_2$ are $3$-rigid then so too is $G$.
A proof of the equivalence of (i) and (ii) in the main theorem can therefore be completed by showing that the uncontractible graphs in $\T$ are $3$-rigid.

\section{The rigidity of uncontractible graphs}

We now show that the uncontractible graphs of $\T$  are $3$-rigid. This completes our first proof of the equivalence of (i) and (ii) in the main theorem. 

This will be achieved in  two steps. The first step shows that uncontractible graphs $G$ in $\T$ have no \emph{interior vertex}. This means that each vertex of $G$ lies on $\partial G$, and so, in particular, $|V(G)| \leq 9$. This leads quickly to the fact that the uncontractible graphs with at most $8$ vertices are $3$-rigid. In the second step we show that the graphs with $9$ vertices and no interior vertices are contractible.

\begin{lemma}\label{l:degree6}
Let $G \in \T$ be an uncontractible graph. Then the interior vertices of $G$ have degree at least $6$.
\end{lemma}

\begin{proof}
A vertex $v$ is an interior vertex if and only if all edges incident to $v$ are of $FF$ type. Since $G$ is $(3,6)$-tight it contains no vertices of degree $1$ or $2$. If $\deg(v)=3$ then it follows from the simplicity of $G$ that each of the three edges incident  to $v$ does not lie on a  non-facial $3$-cycle. This contradicts the uncontractibility of $G$.
If $v$ has degree $4$ then the induced subgraph $X(v)$ for $v$ and its $4$ neighbours has at least $10$ edges. These are the $8$ edges for the faces incident to $v$ and at least $2$ further edges
to fulfil the uncontractibility condition. Thus $f(X(v))\leq 5$ which is contrary to $G$ being $(3,6)$-tight.
One can similarly check that $f(X(v))<6$ if the degree of $v$ is $5$. 
\end{proof}

The proof of the next lemma exploits the topological nature of the torus.
For this and subsequent arguments it is convenient to define the {homology class} of an $FF$ edge in a torus with hole graph whose vertices lie on $\partial G$ and it is convenient to refer to such an edge as a \emph{crossover edge}.

\begin{definition} Let $G$ be a torus with hole graph with triple $(T,D,i)$, let $e$ be a crossover edge and let  $\tilde{e}$ be any cycle of edges formed by $e$ and edges from $\partial G$. The \emph{homology class} of $e$ is 
the unordered pair $\{[\tilde{e}],-[\tilde{e}]\}$ associated with
the homology class of $[\tilde{e}]$ in $H_1(T,\bZ)$. 
\end{definition}

Note that there can be no crossover edge  with trivial homology class. Such an edge would provide a chord of the $9$-cycle $i(\partial D)$ together with an associated facial triangulation with a subpath of the $9$-cycle. This would show that $G$ contains a subgraph which is a torus graph with a single hole where the hole boundary length is less than $9$ and this contradicts $(3,6)$-tightness.

We simply write $[\tilde{e}]$ for the homology class of $e$ and there will be no cause for confusion.

 

Figure \ref{f:smallgraph_again} indicates three sets of crossover edges with the same homology class in the case of the graph $H_1$.
 
\begin{center}
\begin{figure}[ht]
\centering
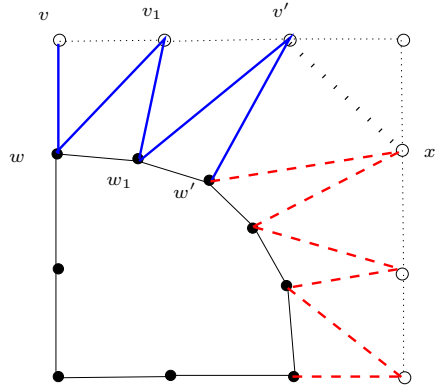
\caption{The $12$ nonboundary edges of the graph $H_1$ fall into $3$  homology classes.}
\label{f:smallgraph_again}
\end{figure}
\end{center}

The limited possibilities for the classes  
 $[\tilde{e}]$  become apparent on considering $G$ as an embedded graph on the topological torus. In the case of boundary type $9v$ the boundary graph edges determine a simple closed curve, $\gamma$ say. If there are no interior vertices then the curves for the remaining edges are disjoint except possibly at their endpoints on $\gamma$. In view of this disjointness 
it follows that there can be at most \emph{three} distinct homology classes for such edges.
Indeed, Figure \ref{f:terminalprooffirst} indicates  three embedded crossover edges (with homology classes $(1,0)$, $(0,1)$ and $(1,1)$, up to sign). Note that no further  homology class is possible for any additional embedded $FF$ edge. (The figure illustrates the embedding of a $9v$ type graph but in fact the argument is the same in general, when $\gamma$ may have points of self contact.)
\begin{center}
\begin{figure}[ht]
\centering
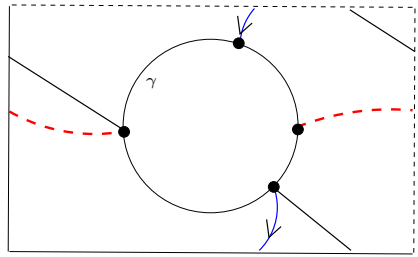
\caption{A representation of $3$ embedded edges of $FF$ type with different homology classes.}
\label{f:terminalprooffirst}
\end{figure}
\end{center}

\begin{lemma}\label{l:atmost9}
Let $G \in \T$ be an uncontractible  graph. Then $G$ has no interior vertices.
\end{lemma}

\begin{proof}Let $(T,D,i)$ be a triple associated with $G$.
Let $v_1,v_2, \dots , v_r$ be the neighbours of an internal vertex $z$ written in order, so that $zv_1v_2, zv_2v_3, \dots $ are facial $3$-cycles of $G$.
By the uncontractibility of $G$ for each vertex  $v_i$ the edge $zv_i$ lies on a non-facial $3$ cycle and so there is an additional edge $v_iv_j$ for some $j\neq i-1, i+1$. 

Suppose first that the degree of $z$ is $6$. Then the subgraph 
$X(z)$ induced by $z$ and its neighbours includes the $6$ edges incident to $z$, the $6$ perimeter edges $v_iv_{i+1}$ and  additional edges between non adjacent perimeter vertices  $v_1, \dots , v_6$.
There are at least $3$ such edges, and since $f(X(z))\geq 6$ it follows that there are exactly $3$, say $e_1, e_2, e_3$.
Let $\tilde{e}_1, \tilde{e}_2, \tilde{e}_3$ be choices of $3$ non-facial $3$-cycles for the edges $e_1, e_2, e_3$ and let $[\tilde{e}_1], [\tilde{e}_2], [\tilde{e}_3]$ be the associated homology classes in
$H_1(T, \bZ)$. 

Figures \ref{f:interiorontorus}, \ref{f:interiorontorusB}  show 
two examples of such a graph $X(v)$ embedded on a topological torus
$S^1 \times S^1$. 
For an appropriate identification of $H_1(T,\bZ)$ with $\bZ^2$ the homology classes $[\tilde{e}_1], [\tilde{e}_2], [\tilde{e}_3]$ in these examples 
are
\[
(1,0),(1,0),(0,1) \quad \mbox{ and } \quad (1,0), (0,1),(1,1). 
\] 
\begin{center}
\begin{figure}[ht]
\centering
\includegraphics[width=7cm]{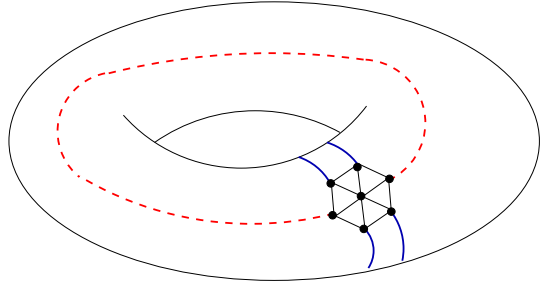}
\caption{An embedding of a subgraph $X(z)$.}
\label{f:interiorontorus}
\end{figure}
\end{center}

\begin{center}
\begin{figure}[ht]
\centering
\includegraphics[width=7cm]{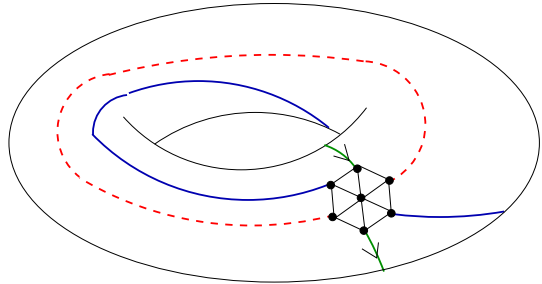}
\caption{An embedding of subgraph $X(z)$.}
\label{f:interiorontorusB}
\end{figure}
\end{center}  

Consider the maximal connected open subsets  $R$ of  $S^1 \times S^1$ that are complementary to $X(z)$ as an embedded graph. We refer to these relatively open sets as \emph{regions}. Examining rectangular representations of the embedded graph we see that in the first example 
there are two non-facial regions, one of which is bounded by a $4$-cycle of embedded edges and one of which is bounded by an $8$-cycle, and there are  $6$ regions for the facial $3$-cycles of $X(v)$.
This also holds for the case of a homology class triple $(1,0), (1,1), (1,1)$.

In the example of Figure \ref{f:interiorontorusB}, with homology triple
$(1,0), (0,1), (1,1)$  there are two non-facial regions each of which is bounded by an embedded $6$-cycle. Thus in both cases that the nonfacial regions are bounded by
$4, 6$ or $8$ edges.

The remaining case in which the homology classes $[\tilde{e}_1], [\tilde{e}_2], [\tilde{e}_3]$ coincide cannot occur since $G$ is a simple graph without loops and the embedded edges for $e_1,e_2, e_3$ are disjoint except at their endpoints.

Consider now the embedded graph for the containing torus graph $T$ for $G$. Note that each of the non-facial regions $R$ defines a subgraph $G_R$ of $T$ which is a torus graph with a single hole whose boundary corresponds to the boundary of the region. In view of the observations on boundary lengths, each graph $G_R$ has freedom number $f(G_R)<6$. 
However, one of these graphs is a subgraph of $G$, and this contradicts the $(3,6)$-sparsity count for $G$. So we conclude that $G$ has no interior vertices of degree $6$.

Suppose now that $\deg(z)=7$. We may view $X(z)$ in this case as arising from the  $\deg(z)=6$ case by the addition of a new vertex $v_*$ between $v_i$ and $v_{i+1}$, with  the edge $v_iv_{i+1}$ replaced by two edges
$v_iv_*, v_*v_{i+1}$ and with the edge $zv_*$ added. 
Moreover an embedded graph in the degree $6$ case can be augmented in a corresponding way to provide the embedded graph for the degree $7$ case.
By the uncontractibility of $G$ there is an edge $v_*v_j$ for some $j$ which is embedded and we see that its embedding divides an $r$-cycle region (for the degree $6$ case)  into two regions. Moreover it follows that the boundary cycles for these regions do not use more edges. Thus we obtain a contradiction as before. By induction the same conclusion holds in general and so $G$ has no interior vertices of degree $n$ for all $n\geq 6$. In view of the previous lemma there can be no interior vertices of any degree.
\end{proof}

Recall  that the double banana graph, $G_{DB}$ say, is formed by joining two copies of the (single banana) graph $K_5\backslash e$ at their degree $3$ vertices. Evidently this graph is $(3,6)$-tight and fails to be $3$-rigid
and it is well known that this is the only graph with $8$ or fewer vertices with this property. The next lemma combined with the previous lemma shows that the uncontractible graphs with $8$ or fewer vertices are $3$-rigid.

\begin{lemma}\label{l:lessthannine}
Let $G\in \T$ be a  graph in $\T$ with no interior vertices and  no more than $8$ vertices. Then $G$ is $3$-rigid.
\end{lemma}

\begin{proof} It suffices to show that if $G$ has $8$ vertices then it is not equal to  $G_{DB}$.  Since $G$ and $\partial G$ have the same vertex set  it follows that the boundary graph is of type $v3v6$ or $v4v5$.

Considering an embedded graph representation of $G$ in a torus $R/\sim$, with the  detached disc represented by an open subset of $R$, it follows that in fact any torus with hole graph with this form of boundary is $3$-connected.  Since the double banana graph is not $3$-connected the proof is complete.
\end{proof}

We now embark on showing the remaining case (step 2) that the graphs $G$ in $\T$ with boundary type $9v$ and no interior vertices are contractible.
It is possible to give a rather extended ad hoc embedded graph argument to show this, and in fact this method is employed in Section 6 for a range of small graph types.
However we now give a more sophisticated proof, exploiting the homology classes of edges, which provides  some general methods and other insights.

The main idea  may be illustrated by considering again the $9$ vertex graph $H_1$, labelled as in Figure \ref{f:smallgraph_again}. 
The edges $vw$ and $v'w'$ are $FF$ edges of the same {homology class}, as are the $3$  "intermediate $v$ to $w$ edges". We note that the "interior edge" $v_1w_1$ (in contrast to $v'w_1$) does not lie on a nonfacial $3$ cycle. Also it can be shown that it does not lie on a critical separating cycle, and so the graph is reducible in $\T$ by edge contraction. In general we will identify such edges within subgraphs (called panel subgraphs) determined by crossover edges of the same homoloy class.

First we note two general lemmas that ensure contractibility. (These lemmas also play a role in  the ad hoc arguments in Section 6.)  

In the next proof we refer to an edge $e$ of $G$ as being \emph{critical} if it lies on a critical separating cycle.

\begin{lemma}\label{l:degree3notcrit}
Let $G$ be a graph of $\T$ with a degree $3$ vertex on the boundary graph $\partial G$ which is incident to an $FF$ edge $e$. Then the graph obtained from the contraction of $e$ is in $\T$.
\end{lemma}

\begin{proof}
Let $v$ be a such a degree $3$ vertex, with $e=vw$ an $FF$ edge, and let $vv_1, vv_2$ be the other edges incident to $v$. Note that both $vv_1$ and $vv_2$ must be edges of $\partial G$. 
Suppose, by way of contradiction, that $vw$ lies on a critical separating cycle $c$. Relabelling $v_1$ and $v_2$ if necessary, we may assume that the edge $v_1v$ also lies on $c$. (See Figure \ref{f:noncritedge}.) 
\begin{center}
\begin{figure}[ht]
\centering
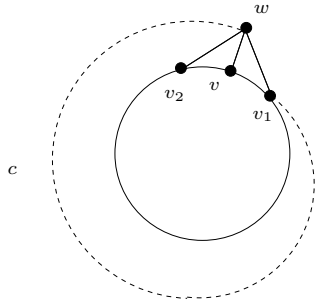
\caption{The noncriticality of the edge $vw$.}
\label{f:noncritedge}
\end{figure}
\end{center}But in this case the  $(3,6)$-tight subgraph $G_1$ determined by $c$ contains a torus with hole graph with boundary cycle of length $8$ obtained by replacing the two edges $v_1v, vw$ in $c$ by $v_1w$, and this is a contradiction.

Note also that $vw$ does not lie on a nonfacial $3$-cycle, since $G$ is a simple graph, and so the key lemma shows that the contraction of $e$ yields a graph in $\T$.
\end{proof}

\begin{lemma}\label{l:notOnCritOrDegree3exists}
Let $G\in \T$ with $V(G)=V(\partial G)$ and let $c$ be a critical cycle of edges which is not the detachment map cycle for $G$. Then there is an edge contraction of $G$ to a graph in $\T$.
\end{lemma}

\begin{proof}
Let $\pi$ be the (possibly improper) $9$-cycle for the detachment map of $G$. The $9$-cycles $c$ and $\pi$ form the boundary of a (possibly degenerate) facially triangulated annular subgraph $A$ of $G$ and we assume that they have the same orientation. Let $x$ be a vertex such that the edges $e=xy, f=xw$ are  edges of $\pi$ and $c$ which start at $x$, with $e \neq f$, and such that there are subpaths $\pi_1$, $c_1$ of the critical cycles from $x$ to a common vertex $z$. We may also assume that $z$ is the first such vertex. Thus the subpaths  form the boundary of either a facially triangulated disc, if $z \neq x$, or, if $z=x$, a triangulated disc with two boundary vertices identified. We denote this subgraph of $A$ as $A_1$.
See Figure \ref{f:panel_lozenge}.

The subpaths $\pi_1, c_1$ are of the same lengths, say $r$, since $\pi$ and $c$ are critical separating cycles, and the triangulation is formed by the addition of edges only. It follows by elementary graph theory   that there is a degree $3$ vertex $u$ strictly between $x$ and $z$ on the subpath $\pi_1$.
Indeed, the graph $A_1$ has exactly $2r-2$ bounded faces  from the triangulation by facial $3$-cycles of $G$. If the degrees of the intermediate vertices $u$ are greater than $3$ then there would be at least $2(r -1)+1$ faces incident to these vertices. 

Since the vertex $u$ is incident to an $FF$ edge the previous lemma applies to complete the proof.
\end{proof}
 
\begin{center}
\begin{figure}[ht]
\centering
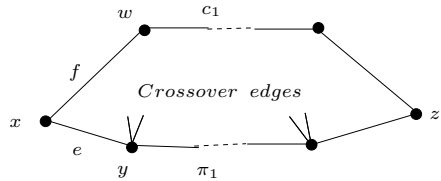
\caption{Subpaths of $c$ and $\pi$ from $x$ to $z$.}
\label{f:panel_lozenge}
\end{figure}
\end{center}

We now define \emph{panel subgraphs} of a graph $G\in \T$ with $|V(G)|=|V(\partial G)|=9$, and show that their strongly interior edges, should they exist, allow contraction to a graph in $\T$.

Consider  two distinct crossover edges $e=vw, e'=v'w'$ which have the same homology class. They determine a triangulated disc subgraph of $G$, and its containing torus graph $T$, which may be visualised as a planar \emph{triangulated panel} of $G$. The  vertex set consists of the vertices of $e, e'$, of which there are $3$ or $4$ in number,
and  vertices on the hole boundary lying on two paths, between $v$ and $v'$ and between  $w$ and $w'$.  Figure \ref{f:toruspanel} indicates such a panel subgraph with $5$ faces. 
\begin{center}
\begin{figure}[ht]
\centering
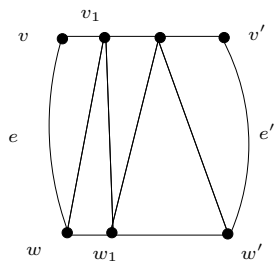
\caption{A panel subgraph of $G$ determined by two crossover  edges $e, e'$ of the same homology class.}
\label{f:toruspanel}
\end{figure}
\end{center}
Considering embedded graphs in the torus it follows readily that every panel graph is contained in a unique maximal panel subgraph.

Note that the \emph{interior vertices  of the panel}, that is, those without incidence with $e$ or $e'$, are only incident to edges of the panel subgraph. It follows that  any \emph{strongly interior edge} $v_1w_1$  of such a panel, that is, one with both $v_1$ and $w_1$ interior vertices, does not lie on a nonfacial $3$-cycle.

Note that Lemma \ref{l:notOnCritOrDegree3exists} implies that if $G$ is not contractible to a graph in $G$ then there can be no
critical separating cycles. Thus, in view of the previous paragraph, there can be no strongly interior edges of the panel if $G$ is not contractible to a graph in$\T$.

\begin{lemma}\label{l:9vhas_stronglyinterior}
Let $G$ be a  graph in $\T$ with $|V(G)|=|V(\partial G)|=9$. Then $G$ may be contracted to a graph in $\T$.
\end{lemma}

\begin{proof}
By the previous discussion we may assume that the maximal panel subgraphs do not have a strongly interior edges. Also, by Lemma \ref{l:degree3notcrit}, we can assume that there are no vertices in $G$ of degree $3$. It follows that each such subgraph has at most $4$ crossover edges. 
Since there are $12$ crossover edges, by $(3,6)$-tightness, and at most $3$ crossover edge homology classes, it follows that there are $3$ panels, each with $4$ crossover edges. Thus the panel subgraphs have the form indicated in Figure \ref{f:smallpanel}.
\begin{center}
\begin{figure}[ht]
\centering
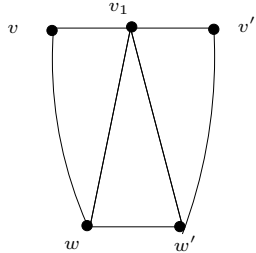
\caption{A panel subgraph with $4$ crossover edges and no interior degree $3$ vertex.}
\label{f:smallpanel}
\end{figure}
\end{center}

By Lemma \ref{l:notOnCritOrDegree3exists} we may assume that  every crossover edge is not on a critical cycle.  However, we now show that this is not possible, completing the proof. 

Consider a single pair of internal edges, $v_1w, v_1w'$ on one of the panel graphs. These edges and their panel subgraph are illustrated in  Figure \ref{f:irreducibleproof}. If  $v_1w$ lies on a nonfacial $3$-cycle then this is achieved through  edge $v_1v'$ of the panel and one of the $8$ remaining crossover edges with a different homology class.  This crossover edge is indicated in the figure with label $g$. 
Note however, that if $g'$ is another embedded edge of the same homology type as $g$ then, from the disjointness requirement,
$g'$ has the form $v'x$ (or the form $yw$), where $x$ (resp. $y$) lies on the arc from $v$ to $w$ (resp. $v'$ to $w'$) as indicated in the figure. It follows that the form of the panel  for this homology type cannot hold, and this contradiction completes the proof.
\end{proof}

\begin{center}
\begin{figure}[ht]
\centering
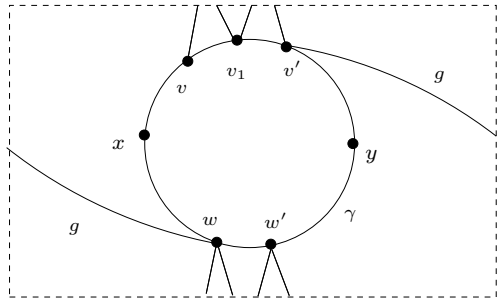
\caption{Only edges of the form $xv'$ or $vy$ have the same homology type as the edge $g$.}
\label{f:irreducibleproof}
\end{figure}
\end{center}

\section{Vertex splitting construction}

We now obtain an alternative proof of the equivalence of (i) and (ii) in the main theorem. Also we determine exactly the uncontractible graphs in $\T$ and with this we complete the proof of the equivalence of (i), (ii) and (iii).

\begin{lemma}\label{l:contractionkey} Every graph in $\T$ admits a contraction sequence in $\T$ to an uncontractible graph.
\end{lemma}

\begin{proof}
Let $G \in  \T$ and suppose that $G$ is a contractible torus with hole graph. We show that there exists a contractible edge for which the contraction yields a graph in $\T$.

Since $G$ is contractible there exists an $FF$ edge which is not on a nonfacial $3$-cycle. If the contracted graph is in $\T$ we may continue the argument with a smaller graph. So we may assume, by the critical cycle lemma, Lemma \ref{l:toruskey}, and the discussion preceding Corollary \ref{c:canconstruct}, that there exists a proper critical cycle $c_1$ and an associated division $G \to \{G_1, G_2^\circ\}$ where $G_1$ is in $\T$ and both $G_1$ and the annular graph $G_2^\circ$ have fewer vertices than $G$. 
Since  $G_2^\circ$ contains a face of $G$ it follows, again by the discussion preceding Corollary \ref{c:canconstruct} that  the annular graph contains an edge $e$ of $FF$ type.

We may assume, moreover, that the edge $e$ does not lie on a nonfacial $3$-cycle. To see this we note the two possibilities that hold if $e$ does lie on a nonfacial $3$-cycle. Either the $3$-cycle is triangulated by faces of $G$, or the triangulation of the $3$-cycle in the containing torus graph for $G$ contains the associated triangulated disc for the detachment map for $G$. In the latter case it follows that $G$ contains a fully triangulated torus graph, which violates $(3,6)$-sparsity, and so this case does not occur. In the former case $G$ contains  the graph of a triangulated sphere all of whose faces are faces of $G$. Such graphs have edges which are of $FF$ type and do not lie on a nonfacial $3$-cycle, and so we may rechoose $e$ to be such an edge.

Since the annular graph contains an edge of $FF$ type which is not on a nonfacial $3$-cycle  we may either contract with this edge to a smaller graph in $\T$ or, by the critical cycle lemma,  obtain another proper critical cycle, $c_2$ say, and an associated division $G \to \{H_1, H_2^\circ\}$ with $H_1 \in \T$, where these graphs have fewer vertices than $G$.

Note that we can assume that $c_2$ lies inside $c_1$, or, more precisely, that the detached triangulated disc for $c_2$ is contained in the detached triangulated disc for $c_1$. One way to see this is to note that 
the union $J$ of $G_1$ and $H_1$ is a proper subgraph of $G$ which lies in $\T$. Thus we may replace $c_2$ by the proper critical cycle for the detachment map for $J$. By the reasoning above there is a contractible edge in the associated annular graph which does not lie on a nonfacial $3$-cycle.

Repeating this process, identifying nested critical cycles $c_1, c_2\dots $,  we must eventually obtain a contractible edge for which the contracted graph lies in $\T$. Indeed, if this did not occur then  the process provides an infinite strictly increasing sequence of proper subgraphs of $G$ and this contradicts the finiteness of the graph.
\end{proof}

We have shown that the uncontractible graphs in $\T$  are $3$-rigid by the analysis in Sections 4 and 5. This, together with
Lemma \ref{l:contractionkey} provides a second proof of the equivalence of (i) and (ii) in the main theorem,
which avoids the use of fission moves.

In fact we can obtain a stronger result by enlarging the analysis of the small graphs given in Sections 3 and 5 to identify the uncontractible  graphs of $\T$.

\begin{thm}\label{t:theirreducibles}
The uncontractible graphs of $\T$ are the graphs $H_{16}$ and $H_{17}$.
\end{thm} 

Note first that $H_{16}$ and $H_{17}$ are uncontractible. Also it is straightforward to check, with the assistance of Lemmas  \ref{l:degree3notcrit} and \ref{l:notOnCritOrDegree3exists}, that the particular graphs $H_2, \dots , H_{15}$ are contractible
in $\T$, and indeed, are completely contractible to one of these two graphs. However, it remains to check that
every graph with no interior vertices and one of these boundary types is similarly contractible. To see that this is so we now argue in a somewhat ad hoc case-by-case manner. The arguments become progressively simpler as the number of vertices decreases. 

The $9v$ case, corresponding to the detachment map $\alpha_1$ is  covered by Lemma \ref{l:9vhas_stronglyinterior}. Alternatively one can effect a proof in the style of the following argument for the $v3v6$ case.

\subsection{The $\alpha_2, \alpha_3$ cases; types $v3v6$, $v4v5$, with $|V(G)|=8$.}

We may assume that there are no vertices of degree $3$ incident to $FF$ edges, in view of Lemma \ref{l:degree3notcrit}. Since the boundary is of type $v3v6$ the vertex $v$ is of degree $4$ or more. Also, every
other vertex is incident to an $FF$ edge (a crossover edge in fact) and so we may assume that  there are no degree $3$ vertices.

We consider the possibilities for an embedded graph representation of such a graph which has been "standardised" so that 

(i) the $v3v$ subcycle of the boundary cycle, with vertex sequence $v, x, y, v$, appears as the right hand boundary of a representing rectangle for $S^1 \times S^1$.

(ii) the first two edges, $vw, wz$ of the $v6v$ subcycle appear on the lower boundary of the representing rectangle for $S^1 \times S^1$.

The embedded representation of $G$ may now be partially indicated, as in Figure \ref{f:v3v6AdHoc}.

\begin{center}
\begin{figure}[ht]
\centering
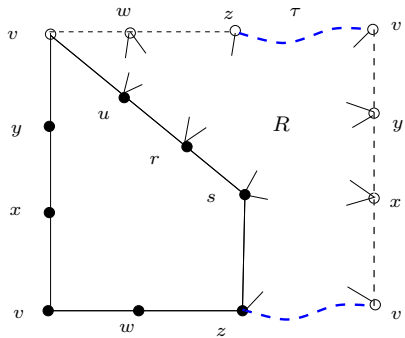
\caption{A partial embedded graph representation of a graph of type $v3v6$ with no interior vertices.}
\label{f:v3v6AdHoc}
\end{figure}
\end{center}

The two paths from $z$ to $v$ in the diagram represent a single path, $\tau$ say, in $S^1 \times S^1$ (and do not necessarily indicate a subpath of edges). The remaining embedded edges of $G$ are representable by paths in the region $R$ which may pass across $\tau$ a number of times.

The degree of  vertices being at least $4$ is indicated on the diagram by two emerging edge paths. Since there are $9$ crossover edges in total, there will be two further emerging directions, but we need not indicate this. For $u$ and $w$ these paths  must emerge properly into the open rectangle region $R$ within the outer boundaries.
The vertex $v$ also has at least $1$ emerging edge but in this case there are different possibilities in the embedded representation for emergence since the embedded edge may also start at the right hand vertices labelled $v$. 

Assume, by way of contradiction that $G$ is uncontractible by edge contraction.
Suppose first that the crossover edge $uw$ exists. We know that it does not lie on a critical cycle, by Lemma \ref{l:notOnCritOrDegree3exists}. Thus the edge $uw$ must  lie on a nonfacial $3$-cycle, since otherwise  $G$ is contractible in $\T$. From the standardised diagram it follows that $uw$ can only be represented by an embedded path, $\pi$ say, from $u$ to the lower representative of $z$. We may also assume that this path lies in $R$. This is a contradiction since, for example, it implies that there exists a hole separating cycle of length less than $9$, namely the $6$ cycle of edges $vz, zw,wv,vx,xy,yv$.

Since the edge $uw$ does not exist the face incident to $vu$ has  edges $va$ and $ua$ with $a \neq w$, and $va$ must be  represented by an embedded curve in $R$ starting from  the upper left representative of $v$. If $a$ is one of $r, s$ or $z$ then we obtain a contradiction as before. On the other hand since $G$ is simple, without loop edges or multiple edges between the same vertices, $a$ is not equal to $v, x$ or $y$, and so in all cases we have the desired contradiction.

There is similar argument for the case $v4v5$.

\subsection{The $\alpha_4$ case; type $e3e4$ with $|V(G)|=7$}
Let $G$ be a graph of this type and assume that $G$ is not contractible. Once again we may partially represent $G$ as an embedded graph
on the torus as indicated in Figure \ref{f:e3e4AdHoc}. By Lemma \ref{l:degree3notcrit} the degree of each of the vertices $z,r,s,x,y$
is at least $4$.  
\begin{center}
\begin{figure}[ht]
\centering
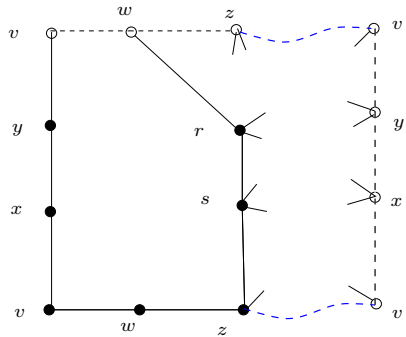
\caption{An embedded graph partial representation of $G$.}
\label{f:e3e4AdHoc}
\end{figure}
\end{center}
In particular there are pairs of edges from each vertex $r, s, y, x$ which are of $FF$ type and which (as before) must lie on nonfacial $3$-cycles. Considering an $FF$ edge out of $y$ the only possibilities are
(i)  $yw$, or (ii) $yz$, which requires $ys$, or (iii) $ys$ which requires either $sv$ or $yz$. Since there are at least $2$ edges out of $y$ it follows that (ii) or case (iii) must hold (even if the degree of $y$ exceeds $4$). 

Case (ii), with $yz$ and $ys$, cannot hold for the following reason. The  triangulation would then require the edge $ry$ and this edge would not lie on a nonfacial $3$-cycle (in any triangulation). Also, similarly, the remaining case (iii), with edges $sy$ and $sv$, would require the edge $sx$ and in any completing triangulation this would not lie on a nonfacial $3$-cycle.
Thus we obtain a contradiction in all cases.

\subsection{The cases $\alpha_5, \dots ,\alpha_9$.}
Suppose, by way of contradiction that there exists an uncontractible  graph where the detachment map has one of these forms, so that there are no repeated edges and exactly two repeated vertices. Since every edge of $\partial G$ lies on a face it follows from
Lemma \ref{l:degree3notcrit},
 that we may assume that each vertex is of degree at least $4$. Since there are $7$ vertices it follows that there are at least $7$ crossover edges. This is a contradiction since there are $9$ noncrossover edges and by $(3,6)$-tightness there are $15$ edges in total.

\subsection{The cases $\alpha_{10}, \dots , \alpha_{15}$.}
Consider a graph $G$ with no interior vertices and detachment map $\alpha_{10}$. In an embedded graph representation on the torus we can assume that
(in the edge identified rectangular representation) the repeated vertex $x$ appears in the corners and the repeated vertices $v$ and $w$ appear on opposite sides. Thus we have the representation in Figure
\ref{f:v1w2x1v2w1x2AdHoc}. In this case note that we can assume that
the boundary of the rectangle is provided by the edges of the $9$-cycle for the hole together with some edge repetitions.

\begin{center}
\begin{figure}[ht]
\centering
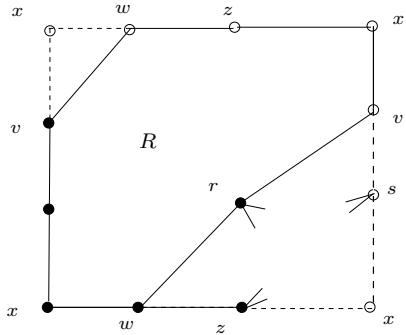
\caption{An embedded graph partial representation for the $\alpha_{10}$ case $v1w2x1v2w1x2$.}
\label{f:v1w2x1v2w1x2AdHoc}
\end{figure}
\end{center}

As before if we assume that $G$ is uncontractible then the degrees of the vertices $r, s$ and $z$ are at least $4$. By $(3,6)$-tightness counting there are $3$ crossover edges. It follows that the degrees are exactly $4$. We see that there is only one triangulation and that $G$ is equal to $H_{10}$ and is contractible. 

The case for type  $\alpha_{11}$ is entirely similar. In fact in each of the subsequent cases it is similarly straightforward to verify that there is a unique triangulation. Since the graphs $H_{12}, \dots , H_{15}$ are contractible the proof Theorem \ref{t:theirreducibles} is complete.


\bibliographystyle{abbrv}

\begin{thebibliography}{30}

\bibitem{cau} A. Cauchy, Sur les polygones et poly\`{e}dres. Second M\'{e}moir. J \'{E}cole Polytechn. 9 (1813) 87-99; Oeuvres. T. 1. Paris 1905, pp. 26-38.

\bibitem{cru-kit-pow} J. Cruickshank, D. Kitson and S.C. Power, 	The generic rigidity of triangulated spheres with blocks and holes, preprint 2014, arXiv:1507.02499.

\bibitem{dehn} M. Dehn, \"{U}ber die starreit konvexer polyeder, Math. Ann. 77 (1916), 466-473.




\bibitem{fog}A.L. Fogelsanger, The generic rigidity of minimal cycles, PhD Thesis, Department of Mathematics, University of Cornell, 1988.

\bibitem{gra-ser-ser} J. Graver, B. Servatius and H. Servatius, Combinatorial rigidity. Graduate Studies in Mathematics,
2. American Mathematical Society, Providence, RI, 1993.

\bibitem{glu} H. Gluck,
{Almost all simply connected closed
surfaces are rigid,} in Geometric Topology, Lecture Notes in
Math., no. 438, Springer-Verlag, Berlin, 1975, pp. 225-239.






\bibitem{whi-vertex}W. Whiteley, Vertex splitting in isostatic frameworks, Structural Topology, 16 (1990), 23-30.


\end{thebibliography}
\def\lfhook#1{\setbox0=\hbox{#1}{\ooalign{\hidewidth
  \lower1.5ex\hbox{'}\hidewidth\crcr\unhbox0}}}

\end{document}